\documentclass[%
 aip,
 amsmath,amssymb,
 reprint,%
]{revtex4-1}

\usepackage{graphicx}
\usepackage{dcolumn}
\usepackage{bm}
\usepackage{placeins}

\usepackage{listings}
\lstdefinestyle{pseudocode}{
    backgroundcolor=\color{gray!10},
    basicstyle=\ttfamily\footnotesize,
    frame=single,
    rulecolor=\color{gray!35},
    columns=fullflexible,
    breaklines=true,
    xleftmargin=0.5em,
    xrightmargin=0.5em
}

\usepackage[utf8]{inputenc}
\usepackage[T1]{fontenc}
\usepackage{mathptmx}
\usepackage{etoolbox}

\usepackage{amsmath,amssymb,amsthm}

\newtheorem{thm}{Theorem}
\newtheorem{lemma}{Lemma}
\newtheorem{corollary}{Corollary}
\newtheorem{definition}{Definition}

\newtheorem{remark}{Remark}
\newtheorem{prediction}{Prediction}

\newcommand{\R}{\mathbb{R}}

\newcommand{\btheta}{\boldsymbol{\theta}}
\newcommand{\bv}{\mathbf{v}}
\newcommand{\bOmega}{\boldsymbol{\Omega}}
\newcommand{\bone}{\mathbf{1}}
\newcommand{\bzero}{\mathbf{0}}
\newcommand{\be}{\mathbf{e}}
\newcommand{\bw}{\mathbf{w}}

\newcommand{\ran}{\operatorname{ran}}

\usepackage{caption,graphicx,subfig,xcolor}

\makeatletter
\def\@email#1#2{%
 \endgroup
 \patchcmd{\titleblock@produce}
  {\frontmatter@RRAPformat}
  {\frontmatter@RRAPformat{\produce@RRAP{*#1\href{mailto:#2}{#2}}}\frontmatter@RRAPformat}
  {}{}
}%
\makeatother
\begin{document}

\preprint{AIP/123-QED}

\title{Predicting the occurrence of  Braess paradox in the synchronization threshold of coupled oscillator systems}
\author{Justin Arches}
\author{Emmanuel Fleurantin}
\author{Matt Holzer*} 
\email{mholzer@gmu.edu}
\author{Siddhant Sood}
\author{Shakeeb Uddin}
 \altaffiliation{Department of Mathematical Sciences, George Mason University, Fairfax, VA 22030}

\date{\today}

\begin{abstract}
We study Braess-type paradoxes in coupled oscillator networks where the addition of an edge to the network leads to an increase in the critical coupling required for the existence of a stable phase-locked synchronized solution.  By introducing a continuous edge-weight parameter and constructing an augmented system that allows application of the Implicit Function Theorem, we derive an explicit first-order formula for the sensitivity of the critical coupling, $K'_c(0)$, to small changes in the network structure.  This formula reveals that an edge is locally Braess if the ordering of the oscillators being coupled is opposite the ordering of their components in the critical eigenvector at the saddle-node bifurcation, providing a simple criterion for Braess paradox to occur.  We then employ this local approximation as a predictor for the occurrence of Braess paradox when the edge is fully incorporated in the network.  This prediction is tested numerically in several classes of random networks, demonstrating both high predictive fidelity and the nonlinear limitations of the local approximation.  We also examine several explicit examples, highlighting possible motifs by which Braess-type paradoxes occur. Together, these results establish a predictive framework for understanding how incremental changes in network topology can produce counterintuitive synchronization outcomes.

\end{abstract}

\maketitle

\begin{quotation}
In oscillator networks described by the Kuramoto model, adding a single new edge can \emph{hinder} synchronization rather than promote it, with consequences ranging from power grid stability to biological rhythms. This is often called a Braess-type paradox in reference to the well-known fact that adding a new road to a transportation network can sometimes decrease the efficiency of traffic flow.  We study this phenomenon in the Kuramoto model and derive an analytical predictor, rooted in the Implicit Function Theorem, that classifies any candidate edge as beneficial or Braess directly from the geometry of the current synchronous state. The result is a fast, mathematically grounded tool for understanding how local structural changes drive global, and sometimes counterintuitive, network behavior.

\end{quotation}

\section{Introduction} 

Braess paradox refers to the counterintuitive phenomenon in which adding an edge to a network can degrade its overall performance. Originally 
identified in transportation networks \cite{braess1968paradoxon}, where 
the addition of a road can paradoxically increase the total travel time of all 
drivers, the effect was later shown to be widespread in a broad class 
of networked systems \cite{steinberg1983prevalence}. The phenomenon arises 
not from poor design, but from rational individual behavior that, in 
aggregate, leads to suboptimal global outcomes, a hallmark of the 
tension between local and global optimization in complex networks.

This counterintuitive behavior extends beyond transportation to 
oscillatory networks, where the relevant measure of performance is 
synchronization rather than travel time. Power grids are a 
particularly consequential setting for this phenomenon: generators and 
consumers are represented as nodes with oscillatory dynamics, and the 
stable operation of the grid depends on maintaining phase-locked 
synchronization across the network \cite{witthaut2012braess}. The loss 
of synchronization in such systems can trigger cascading failures, making 
the question of how network topology shapes synchronizability not merely 
of theoretical interest, but of direct practical importance.

We study the Kuramoto model\cite{kuramoto1984chemical}, a canonical framework for 
studying synchronization in coupled oscillator systems.  In this model, 
each node $i = 1, \dots, n$ represents an oscillator whose phase 
$\theta_i(t)$ evolves according to
\begin{equation}\label{eq:kuramoto}
    \frac{d\theta_i}{dt} =\omega_i+K\sum_{j=1}^n A_{ij}\sin(\theta_j-\theta_i),
\end{equation}
where $\omega_i$ is the natural frequency of the $i$th oscillator and $K$ is the coupling constant.  Interactions between oscillators are described by the adjacency matrix $A$ associated with a connected, 
undirected graph $G=(V,E)$.  Each node in the graph is identified with an integer $\{1,2,\dots,n\}$ and edges are represented uniquely by ordered pairs $(p,q)$ with $p<q$.  The adjacency matrix $A$ therefore satisfies $A_{ij} = A_{ji}$ with $A_{ij}=1$ if and only if $(i,j)\in E$.   We adopt the convention to label nodes in the graph in the order of increasing natural frequency.  That is, we order nodes in the graph so that $i>j$ implies $\omega_i\geq \omega_j$.  By potentially transforming to a rotating frame, we may assume $\sum \omega_i=0$.

The evolution of oscillators in the Kuramoto model (\ref{eq:kuramoto}) is subject to two competing influences. First, each oscillator has its own individual natural frequency, which tends to separate oscillators. At the same time, the nonlinear coupling promotes oscillators to match each other's phases. The competition between these two effects plays out as follows.  When $K\ll 1$, the oscillators move freely, and each oscillator rotates at nearly constant angular rate equal to its natural frequency. Conversely, when $K\gg 1$, the nonlinear terms dominate, and there exists a stable steady state solution where all phases are approximately equal. These two regimes are separated by a critical coupling constant $K_c$, at which this steady state solution disappears through a saddle-node bifurcation. The study of this critical coupling constant and its dependence on the set of natural frequencies and the network of interactions has been an immense field of research over the previous decades \cite{acebron2005kuramoto,dorfler2011critical,rodrigues2016kuramoto}.

We are concerned with the following problem.  Suppose that an edge is added to a network.  This leads to additional synchrony-enhancing terms appearing on the right-hand side of (\ref{eq:kuramoto}).  It is natural to expect that the addition of these terms should serve to decrease the critical coupling required for synchronization.  Braess paradox refers to situations where the opposite occurs, meaning the critical coupling increases instead.  The purpose of this article is to study what properties of the steady state solution and the network of interactions would facilitate such a behavior.

We now make some of these ideas precise.  The central object of study is the \emph{phase-locked} (or steady) state, 
in which all oscillators rotate at a common frequency, and their phase 
differences remain constant in time. Such states are solutions to the 
algebraic system
\[
\omega_i + K\sum_{j=1}^n A_{ij}\sin(\theta_j - \theta_i) = 0, 
\qquad i = 1, \dots, n,
\]
which we write compactly as $F(\boldsymbol{\theta}, K) = 0$. Phase-locked 
states do not exist for arbitrarily small coupling; synchronization 
emerges only when $K$ is sufficiently large to overcome the heterogeneity 
in natural frequencies.  When multiple stable phase-locked states exist, we are interested in the 
one exhibiting the highest degree of synchrony, as measured by the order 
parameter
\begin{equation} \label{eq:orderpar}
r(\mathbf{\theta})=\frac{1}{n} \left| \sum_{i=1}^n e^{\mathbf{i}\theta_i}\right|. 
\end{equation}

The order parameter $r \in [0,1]$ quantifies global phase coherence: 
$r = 1$ corresponds to perfect synchronization and $r = 0$ to complete 
incoherence. As we said previously, for $K\gg 1$, one can show that there exists a stable synchronous state with $\theta_i\approx 0$ for all oscillators.  This stable state can be continued in the parameter $K$ until a bifurcation occurs.  This bifurcation is generically a saddle-node, and we define the \emph{critical coupling constant} $K_c$ as the value of $K$ at which this branch of synchronous solutions undergoes a saddle-node bifurcation.  We note that other stable synchronous solutions may co-exist with the branch considered here; see the twisted state solutions of Wiley, Strogatz and Girvan\cite{wiley2006size} for example. However, these alternate steady states typically have a lower degree of synchrony as measured by the order parameter $r$. 

\begin{figure}[h]
    \centering
    \includegraphics[width=0.49\textwidth]{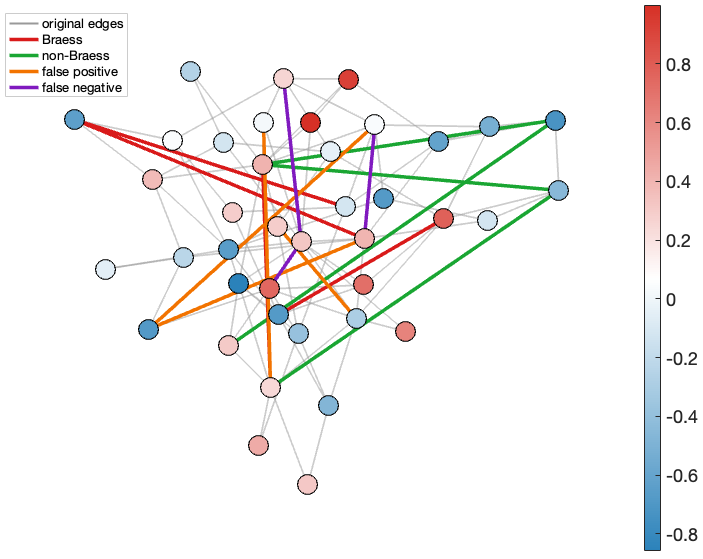}
    \caption{A representative Erd\H{o}s--R\'enyi random network of 
    $n = 40$ oscillators, with nodes colored by natural frequency 
    $\omega_i$ (blue: negative, red: positive). Highlighted edges 
    illustrate the range of behaviors that can arise when a new edge 
    is added to the network: some additions raise the critical coupling $K_c$ and hinder synchronization (Braess edges, red), while others lower $K_c$ and facilitate it (non-Braess edges, green). Misclassified edges (false positives, orange; false negatives, purple) and the original network edges (gray) are also shown.}
    \label{fig:network_archetypes}
\end{figure}

Braess paradox in oscillator networks is generally referred to as a decrease in the level of synchronization in a system after the incorporation of additional coupling terms.  Witthaut and Timme \cite{witthaut2012braess} first demonstrated this phenomenon in power grid models, showing that adding edges can destabilize phase-locked states and reduce synchronization, and establishing the presence of Braess paradox through systematic examination of how individual edge additions alter the critical coupling required for stable operation. Subsequent research in this area has focused on various measures of synchronization. Specifically, upon adding or strengthening an edge in the network, one could view Braess-type phenomena as being associated with 
i) lower degree of synchrony as measured by $r(\theta)$, 
ii) larger loads occurring over individual edges in the network, or 
iii) a decrease in the stability of a phase-locked solution as measured by the value of the least stable eigenvalue appearing in the linearization near the steady state. 
Arola-Fernandez, Skardal, and Arenas\cite{arola2021geometric} studied how edge additions could lead to a lower degree of synchrony as measured by $r(\theta)$.  They expressed the synchronized state through the Laplacian pseudo-inverse, enabling a geometric 
analysis of how network structure shapes phase configurations.  Manik, Witthaut, and 
Timme \cite{manik2017cycle,manik2022predicting} analyzed how infinitesimal perturbations in edge strength alter 
global flow patterns in supply and transport networks, providing a 
framework for identifying edges whose strengthening can induce 
Braess-type behavior through an overloading of certain edges.  Coletta and Jacquod\cite{coletta2016linear} studied how edge addition can lead to a decrease in the stability of a phase-locked solution, as measured by the value of the least stable eigenvalue appearing in the linearization near the steady state, with a primary focus on chain networks. While not phrased explicitly in terms of Braess paradoxes, work by Motter, Zhou and Kurths\cite{motter2005enhancing} showed how heterogeneity in network connections could decrease synchronization.  Taylor, Skardal and Sun\cite{taylor2016synchronization} study optimization of synchronization through the derivation of a synchrony alignment function meant to quantify the interplay between network structure and individual frequencies.  Together, these works make clear that the relationship between
the network's connectivity and synchronizability is subtle and non-monotone.

\begin{figure}[t]
    \centering
    \includegraphics[width=0.95\columnwidth]{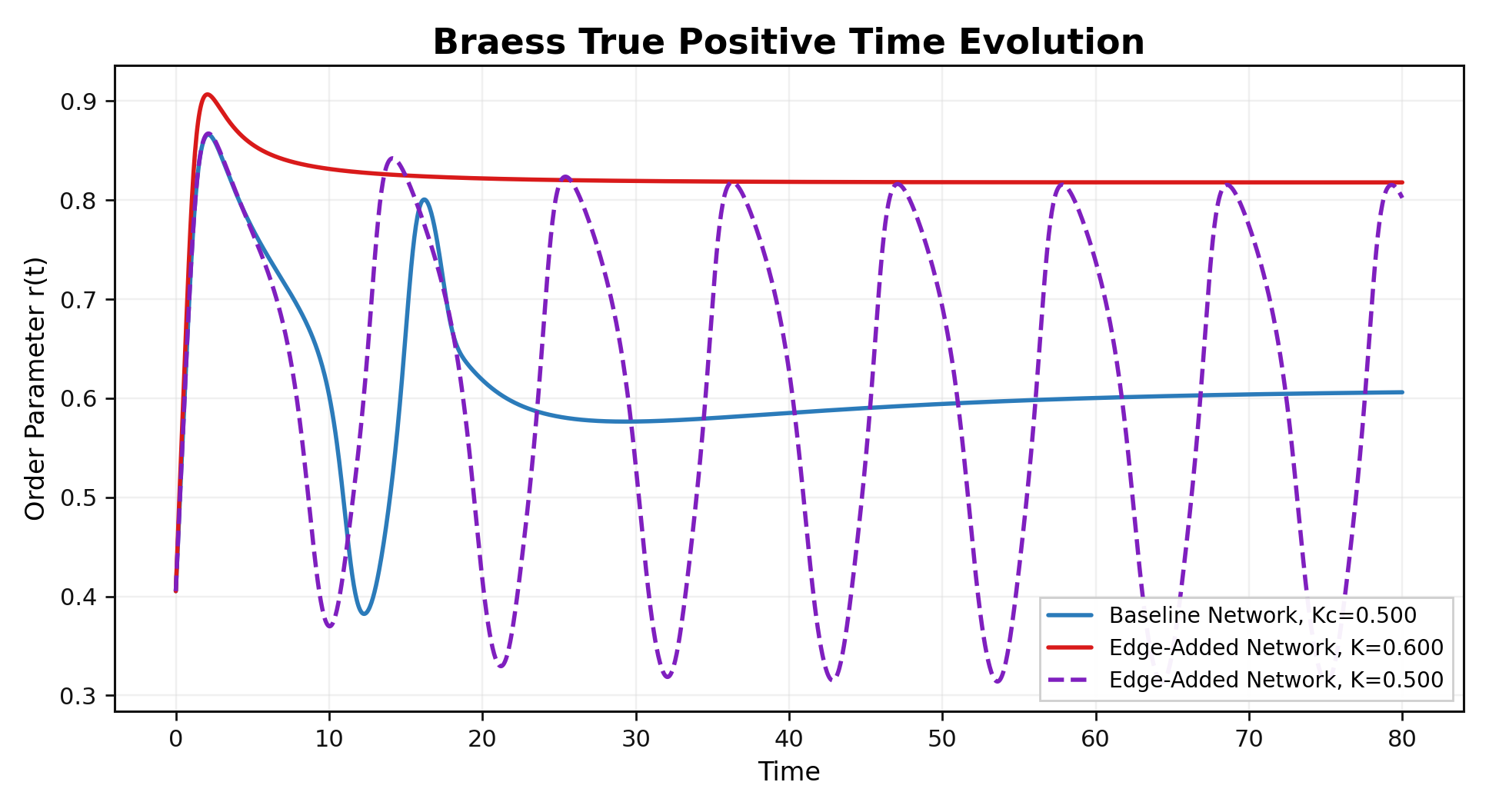}
    \caption{Time evolution of the order parameter demonstrating possible deleterious effect of edge addition on system behavior.  In the baseline network (blue curve), the solution converges to a phase-locked steady state with a constant order parameter indicating a high level of synchrony. We demonstrate the effect of the addition of a Braess edge, where the steady state disappears, and the order parameter oscillates (purple curve) for the same choice of the coupling constant before edge addition.  Increasing the coupling constant is required to recover synchronous behavior (red curve).    }
    \label{fig:true-positive-time-evolution}
\end{figure}

In this work, we adopt a slightly different approach and characterize Braess paradox in terms of how the critical 
coupling $K_c$ changes under a small structural perturbation of the 
network, specifically, the addition of a single new edge; see Figure~\ref{fig:true-positive-time-evolution} for an example potential impact.  Our main result (Theorem~\ref{thm:saddlenode} in Section~\ref{sec:theorem}) is a derivation of an analytical formula, based on the Implicit Function Theorem, 
that predicts the sign of this change. Our framework applies to general 
oscillator networks and allows one to predict whether a candidate edge is 
\emph{Braess} (edge addition raises $K_c$, hindering synchronization) or 
\emph{non-Braess} (edge addition lowers $K_c$, facilitating 
synchronization) without requiring full recomputation of the phase-locked 
state for every edge under consideration. The primary insight gained from our analysis is that Braess edges arise primarily when they couple oscillators whose ordering of phases is opposite the ordering of the respective entries in the critical eigenvector at the saddle-node bifurcation.  We analyze the accuracy of this 
classifier, including cases of false positives and false negatives, and 
identify network archetypes in which behavior type arises.  A representative example of the phenomena studied in this work is shown in Figure~\ref{fig:network_archetypes}, where candidate edges on an Erd\"os R\'enyi random network are classified according to their effect on the critical coupling $K_c$; our predictor correctly identifies most edges, though occasional misclassifications (false positives in orange, false negatives in purple) reveal the boundaries of the linear approximation. A secondary result (Theorem~\ref{thm:stablepert} in Section~\ref{sec:theorem}) computes how the order parameter varies with small changes in the network.  This result requires a full understanding of the spectrum of the steady state solution. However, we apply it in some degenerate cases and show that it recovers the condition derived in  Arola-Fernandez, Skardal, and Arenas\cite{arola2021geometric} in the limit as $K\to\infty$.

The remainder of this paper is organized as follows. 
In Section~\ref{sec:theorem}, we introduce the mathematical framework 
and state the main results. Specifically, we define the augmented system 
$F(\boldsymbol{\theta}, K, a)$ that encodes the addition of a new edge 
$(p,q)$ with weight $a$, and formally state Theorem~\ref{thm:saddlenode}, 
which provides an explicit formula for $K_c'(0)$ via the Implicit 
Function Theorem, yielding a computable criterion for predicting whether 
a candidate edge is Braess. We complement this with 
Theorem~\ref{thm:stablepert}, which characterizes how the order 
parameter $r$ responds to the edge addition near a stable synchronous 
state, and we discuss the geometric interpretation of both results. 
In Section~\ref{sec:numerics}, we test the validity of the analytical prediction 
of Theorem~\ref{thm:saddlenode} against numerical simulations on planar 
random geometric graphs of varying scale, examining the classifier 
performance, including rates of true and false positives and negatives, and identifying the role of coupling oscillators whose phase differences exceed $\pi$ as a source 
of prediction error. In Section~\ref{sec:motifs}, we analyze several 
network motifs in which Braess behavior can be understood explicitly, providing 
structural intuition for the conditions under which the paradox arises. 
Section~\ref{sec:discussion} concludes with a discussion of the results 
and directions for future work. Proofs of 
Theorems~\ref{thm:saddlenode} and~\ref{thm:stablepert} are collected 
in Appendices~\ref{appendixA} and~\ref{appendixB}, respectively.

\section{Main Theorem: Braess edge prediction via Implicit Function Theorem}\label{sec:theorem}

Our goal is to calculate how our critical coupling $K_c$ changes upon adding a new
edge of weight $a$ between two nodes $p$ and $q$. So, we define
a parameter $a\in[0,1]$ where $a=0$ represents the lack of an edge between node
$p$ and $q$, while $a=1$ represents the full addition of the edge. Considering $K_c$ as a function of $a$, we want
to differentiate and compute $K_c'(0)$.  We first
redefine our steady state equation to incorporate the new addition of an edge.

\begin{definition}\label{def:steadystate}
  Let $F:\R^n\times\R\times\R\to \R^n$ such that
  \begin{equation} F(\boldsymbol{\theta}, K, a) = \boldsymbol{\Omega} +
    K(s(\boldsymbol{\theta})+ag_{pq}(\boldsymbol{\theta}))\label{eq:F} \end{equation}
  where $\boldsymbol{\Omega}=(\omega_1,\dots,\omega_n)^T$,
  \[s(\theta)_i = \sum_{j=1}^n A_{i,j}\sin(\theta_j - \theta_i),\]
  and
  \[g_{pq}(\theta) = \sin(\theta_q-\theta_p)(\mathbf{e}_p-\mathbf{e}_q).\]
  We specify for uniqueness that $p<q$, and 
  $\mathbf{e}_k\in\R^n$ represents the $k$-th standard basis vector.
  We define the \textbf{steady state} equation when
  $F(\boldsymbol{\theta}, K, a) = 0$.
\end{definition}

Suppose that when $a=0$, there exists a steady state solution $\boldsymbol{\theta}_c$ that undergoes a saddle-node bifurcation
at $K_c$. At this bifurcation, the Jacobian, 
$D_{\boldsymbol{\theta}}F(\boldsymbol{\theta}_c, K_c, 0)$ has a double zero eigenvalue.  One eigenvector is the constant vector $\mathbf{1}$, which always lies in the kernel of the Jacobian since it has a row sum of zero.  We denote the second eigenvector $\mathbf{v}_c$, which is orthogonal to $\mathbf{1}$. We can then derive an explicit first-order
formula for the critical coupling strength, $K_c'(0)$, as the edge weight is increased incrementally from $a=0$.  Our main result is the following.

\begin{thm} \label{thm:saddlenode}
Suppose that when $a=0$, there exists $(\boldsymbol{\theta}_c,K_c,0)$ such that
$F(\boldsymbol{\theta}_c,K_c,0)=0$ and its Jacobian, 
$D_{\boldsymbol{\theta}}F(\boldsymbol{\theta}_c,K_c,0)$, has a null vector $\mathbf{v}_c$ satisfying $\mathbf{v}_c^T\mathbf{v}_c=1$, $\mathbf{v}_c^T\bone=0$, and
\[\ker\!\left(D_{\boldsymbol{\theta}}F(\boldsymbol{\theta}_c,K_c,0)\right)
=\operatorname{span}\{\bone, \mathbf{v}_c\}.\]

Then, for $a$ near $0$, there exists smooth functions $\boldsymbol{\theta}(a)$, $\mathbf{v}(a)$, $K_c(a)$ such that 
\[\boldsymbol{\theta}(0)=\boldsymbol{\theta}_c, \qquad \mathbf{v}(0)=\mathbf{v}_c, 
\qquad K_c(0)=K_c.\]
and $F(\boldsymbol{\theta}(a), K_c(a), a)=0$ with 
$D_{\boldsymbol{\theta}}F(\boldsymbol{\theta}(a),K_c(a),a)\mathbf{v}(a)=\mathbf{0}$
for all $a$ near $0$. Moreover, if $\mathbf{v}(0)^T\Omega\neq 0$ then 
\begin{equation} \label{eq:Kprime}
    K'_c(0) =
\sin\left(\theta_q(0)-\theta_p(0)\right)(K_c(0))^2\frac{v_p(0)-v_q(0) }{\mathbf{v}(0)^T \Omega }.
\end{equation}
\end{thm}

\begin{proof}
The first step of the proof is to show persistence of the saddle-node bifurcation to synchrony, assumed to hold when $K=K_c$ and $a=0$.  Towards that end, we consider the following augmented system:
\begin{equation} \label{eq:G} G(\boldsymbol{\theta}, \mathbf{v}, K, a) =
\begin{pmatrix}
F(\boldsymbol{\theta}, K, a) \\
D_{\boldsymbol{\theta}}F(\boldsymbol{\theta}, K, a)\,\mathbf{v} \\
\mathbf{v}^{\top} \mathbf{v} - 1 \\
\mathbf{v}^{\top} \mathbf{1} \\
\theta_1
\end{pmatrix}.
\end{equation}
We will show that there exists a solution to $G(\boldsymbol{\theta}, \mathbf{v}, K, a)=\mathbf{0}$ for all $a$ sufficiently small.  This solution will be the continued saddle-node bifurcation to synchrony.  The first $n$ equations in (\ref{eq:G}) enforce a steady state solution to (\ref{eq:kuramoto}).  The following $n$ equations enforce the existence of a kernel element to the linearization.  The last three conditions normalize this kernel element to be a unit vector, enforce that this vector is orthogonal to the kernel element $\mathbf{1}$, while the final phase condition removes the phase invariance in (\ref{eq:kuramoto}).  

The existence  of a curve $(\mathbf{\theta},\mathbf{v},K)(a)$ with $\boldsymbol{\theta}(0)=\boldsymbol{\theta}_c$, $\mathbf{v}(0)=\mathbf{v}_c$, 
$K_c(0)=K_c$, for which $G(\mathbf{\theta}(a),\mathbf{v}(a),K_c(a),a)=0$ follows from a Liapunov-Schmidt reduction; see Appendix~\ref{appendixA} for details. 

To obtain (\ref{eq:Kprime}), we expand and differentiate the following scalar equation with respect to $a$ and evaluate at $a=0$:
\begin{equation}
    \mathbf{v}(0)^T\cdot F(\mathbf{\theta}(a),K_c(a),a)=0,
\end{equation}
see Appendix~\ref{appendixA} for details. 
\end{proof}

Theorem~\ref{thm:saddlenode} establishes a criterion dictating whether the critical constant $K_c$ increases or decreases when the weight of edge $(p,q)$ is increased by some small amount.  Since one expects additional connectivity to reduce the synchronization threshold, an increase in $K_c$ is seen as counterintuitive and categorizes such edges as Braess edges in the following manner.  

\begin{definition} 
    The edge $(p,q)$ is a Braess edge if $K_c(1)>K_c(0)$. The edge $(p,q)$ is locally Braess if $K_c'(0)>0$.    
\end{definition}

In the context of (\ref{eq:kuramoto}), the edge $(p,q)$ is a Braess edge if the critical coupling is larger when the edge $(p,q)$ is included in the adjacency matrix $A$.  Theorem~\ref{thm:saddlenode} provides a local description of whether $K_c$ is increasing or decreasing for small edge weight $a$.  This local behavior can be used as a prediction for the critical coupling when $a=1$ and the edge is fully incorporated in the network.  We formulate this as the following prediction.

\begin{prediction}
    The edge $(p,q)$ is a Braess edge if $K_c'(0)>0$.  
\end{prediction}

We use the phrasing {\em prediction} as opposed to conjecture because we do not expect this to hold universally for all edges.  Numerical investigations of the accuracy of this prediction in some sample networks are presented in Section~\ref{sec:numerics}.   Here, one can consider the critical coupling constant as a function of the added node weight $a$.  A locally Braess edge will necessarily satisfy that the graph of $K_c(a)$ will be locally increasing near $a=0$.   Theorem~\ref{thm:saddlenode} further denotes the only two generic ways in which this $K_c(a)$ graph can lose non-monotonicity: 
either the phases of the coupled oscillators will match for some value of $0<a<1$ 
or the respective values in the critical eigenvector will.  
Interestingly, our numerics reveal that it is typically the case that edges which are locally Braess but not Braess (false positive predictions) will see that the phases of the respective oscillators to swap at some intermediate value of $a$, whereas edges which are Braess but not locally Braess (false negatives) will see that the respective entries in the critical eigenvector will swap. See examples in Section~\ref{sec:numerics} and Section~\ref{sec:motifs} for more details.

Suppose that the phase difference between the two oscillators $p$ and $q$  is less than $\pi$ in magnitude.  Theorem~\ref{thm:saddlenode} gives a simple criterion to identify whether the edge $(p,q)$ is locally Braess: if the oscillator phases' ordering is opposite to that of the respective elements' ordering of the entries in the (normalized) critical eigenvector, then the edge is locally Braess.  To make this precise, we specify a unique phase for each oscillator as the unique phase (subject to the phase condition $\theta_1=0$) for which $\lim_{K\to\infty} \theta_j(K)=0. $

One might also be interested in the behavior of the order parameter $r$ when a new edge is added to the network. The following results give the derivative of $r$ as a function of the parameter $a$.

\begin{thm}\label{thm:stablepert} Suppose that there exists a stable synchronous steady state solution $\mathbf{\theta}^*$.   Denote the eigenvalues of the linearization by $\lambda_k$ with the assumption that $0=\lambda_1>\lambda_2\geq \lambda_3\geq \dots\geq \lambda_n$ with corresponding orthonormal eigenvectors $v_k$. Then there exists a steady state solution $\theta^*(a)$ and the order parameter $r(\theta^*(a))$ is locally a smooth function of $a$ whose derivative satisfies: 
\small{
\begin{equation} \label{eq:ravaries} \frac{d}{da} (\frac{r(\theta^*(a))^2}{2}) =-\frac{K_c(0)}{n^2}\sin\left(\theta^*_q(0)-\theta^*_p(0)\right) 
\sum_{k=2}^n \gamma_k\left(\frac{\mathbf{v}_{k,p}(0)-\mathbf{v}_{k,q}(0)}{\lambda_k}\right)\end{equation}
}
where 
\begin{equation} \gamma_k = \mathbf{w} \cdot \mathbf{v}_k, \qquad 
w_j = S\cos\theta_j^* - C\sin\theta_j^*, 
\end{equation}
and
\begin{equation}
C = \sum_{j=1}^n \cos\theta_j^*, \qquad S = \sum_{j=1}^n \sin\theta_j^*.
\end{equation}
\end{thm}
\begin{proof}
    The proof follows by fixing the rotational symmetry and applying the Implicit Function Theorem and is presented in Appendix~\ref{appendixB}.
\end{proof}

Theorem~\ref{thm:saddlenode} and Theorem~\ref{thm:stablepert} generalize immediately to Kuramoto models on weighted graphs described by the system of equations:
\begin{equation}\label{eq:weightedkuramoto}
    \frac{d\theta_i}{dt} =\omega_i+K\sum_{j=1}^n W_{ij}\sin(\theta_j-\theta_i).
\end{equation}
The only difference between (\ref{eq:weightedkuramoto}) and (\ref{eq:kuramoto}) is that the matrix $W_{ij}$ can prescribe arbitrary positive weights to edges as opposed to the unitary weights prescribed by the non-zero entries of $A$.

\begin{remark}
    In the limit as $K\to K_c$ from above, we have that $\lambda_2\to 0$; see the statement of Theorem~\ref{thm:stablepert}.  Observe that in this limit, the dominant contribution guiding whether the order parameter is increasing or decreasing is exactly the product of $\sin(\theta_q^*-\theta_p^*)$ and the difference $v_{2,p}-v_{2,q}$.
\end{remark}

\begin{remark}
    In the limit as $K\to\infty$, we recover the criterion developed by Arola-Fernandez, Skardal, and Arenas\cite{arola2021geometric}; see Appendix~\ref{apc}.
\end{remark}

\section{Numerical investigations on random networks}
\label{sec:numerics}

We tested the prediction of Theorem~\ref{thm:saddlenode} on randomly generated oscillator networks. For a missing edge
$e=(p,q)$, define
\[
W^{e,a}=W+a(\mathbf e_p\mathbf e_q^T+\mathbf e_q\mathbf e_p^T),
\qquad 0\leq a\leq 1,
\]
so that $a=0$ is the original network and $a=1$ is the network after the full edge has been added. In the numerical computations, an edge was treated as Braess when
\[
\Delta K_c^{\mathrm{num}}(e)
=
K_c^{\mathrm{num}}(W^{e,1})-K_c^{\mathrm{num}}(W)>0.
\]
Here $K_c^{\mathrm{num}}$ denotes the numerically estimated synchronization threshold described below. The local prediction was obtained from the sign of the derivative in Theorem~\ref{thm:saddlenode},
\[
K_c'(0)=
\sin(\theta_q-\theta_p)K_c^2
\frac{v_p-v_q}{\mathbf v^T\Omega}.
\]
Thus $K_c'(0)>0$ predicts a Braess edge and $K_c'(0)<0$ predicts a non-Braess edge. In the implementation, the critical eigenvector is always oriented so that $\mathbf v^T\Omega>0$.

\paragraph{Network ensembles.}
We considered two different random graph models.  For \mbox{Erd\H{o}s--R\'enyi} graphs, we repeatedly sampled $G(n,p)$ with $p=0.35$ until the graph was connected. For planar random geometric graphs, node positions were sampled independently in $[0,1]^2$, and two nodes were connected whenever their Euclidean distance was at most $r=0.35$; this process was also repeated until the graph was connected. Natural frequencies were sampled independently from \texttt{Unif[-1,1]} and then shifted so that they had a mean of zero.  The phase gauge was fixed by setting $\theta_1=0$.

\paragraph{Computing the numerical threshold.}
The code does not compute an exact saddle-node bifurcation point. Instead, it estimates $K_c$ using numerical continuation. For each network, the solver begins at $K_{\max}=5$ and employs secant continuation to follow this branch of stable phase-locked solutions until stability is lost in a saddle-node bifurcation and the critical coupling constant is identified as the minimal constant $K$ along this branch of solutions.  

\paragraph{Filtering and evaluation.}
Our numerical investigations computed $\Delta K^{\mathrm{num}}_c(e)$ for ensembles of networks of Erd\"os-R\'enyi and random geometric graph types.  In each case, we compute the base critical value and the value of $K_c'(0)$ from Theorem~\ref{thm:saddlenode}.  To limit computational time, we only consider edges for which the predicted relative deviation in the critical coupling constant exceeds $0.1\%$.

Aggregated data over 20 different random geometric networks with $14$ nodes is shown in Figure~\ref{fig:RGNunfiltered}.  Edges are classified as true/false-positive/negative depending on whether our prediction based upon $K_c'(0)$ is accurate and if it predicts the edge to be Braess or not.  We report a relatively high number of false positive cases, which we attribute to two sources.  First, we see a significant number of false positive cases relating to the coupling of oscillators with a variety of frequencies, nearly between $-\pi$ and $\pi$.  Closer investigation reveals that most of these occurrences involve edges that are only weakly locally Braess in the sense that the predicted relative change in the critical coupling constant is small, i.e., close to the $0.1\%$ cutoff used in the numerical process.  We therefore refine our prediction to only consider edges with a predicted relative deviation exceeding $1\%$.  This data is reported in Figure~\ref{fig:RGN1420}.

\begin{figure}[t]
    \centering
    \includegraphics[width=0.95\columnwidth]{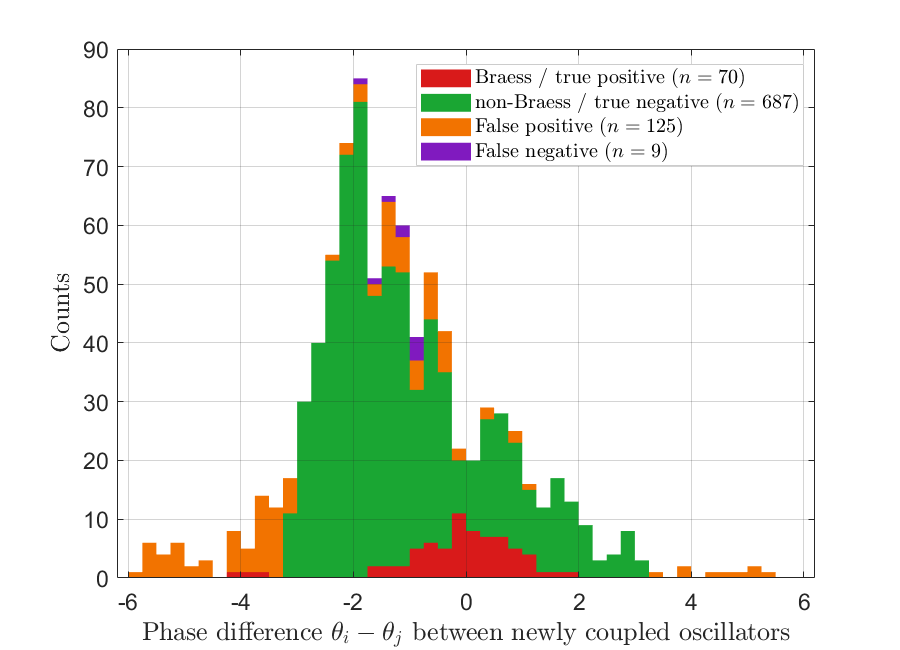}
    \caption{Histogram of Braess predictor outcomes relative to the phase difference between the coupled oscillators for $20$ realizations of random geometric graphs on $14$ nodes.  We consider only nodes where the relative change predicted by Theorem~\ref{thm:saddlenode} is greater than $0.1\%$.   Note that most edges involving the coupling of oscillators with phase differences greater than $\pi$ are not Braess, but are locally Braess.  See Figure~\ref{fig:RGN1420} for the data after filtering to only consider edges with predicted changes exceeding $1\%$.  }
    \label{fig:RGNunfiltered}
\end{figure}

The data in Figure~\ref{fig:RGN1420} shows that restricting consideration to only those edges with relative change exceeding $1\%$ improves our prediction for moderate phase differences, but false positive outcomes persist for edges that connect oscillators with phase differences exceeding $\pi$ in magnitude.  The reason for this will be studied in more detail in some example networks in Section~\ref{sec:motifs}.   A scatter plot of predicted versus observed changes in the critical coupling constants is also provided in Figure~\ref{fig:RGN1420}.  We note somewhat linear behavior associated with the data near the origin.  As expected, there are many more edges whose addition contributes a greater degree of synchronization (true negative cases), and the addition of some edges can lead to quite large decreases in the critical coupling values. 
\begin{figure*}[htbp]
    \centering
    \begin{minipage}[t]{0.49\textwidth}
        \centering
        \includegraphics[width=\textwidth]{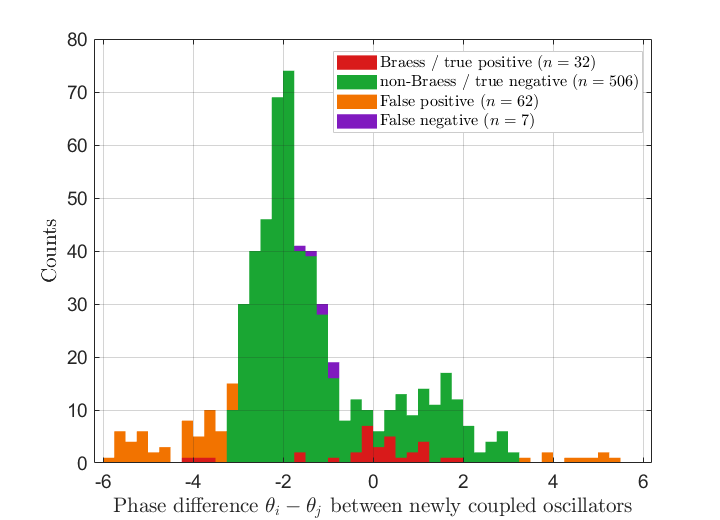}
    \end{minipage}
    \hspace{0.03\textwidth}
    \begin{minipage}[t]{0.42\textwidth}
        \centering
        \includegraphics[width=\textwidth]{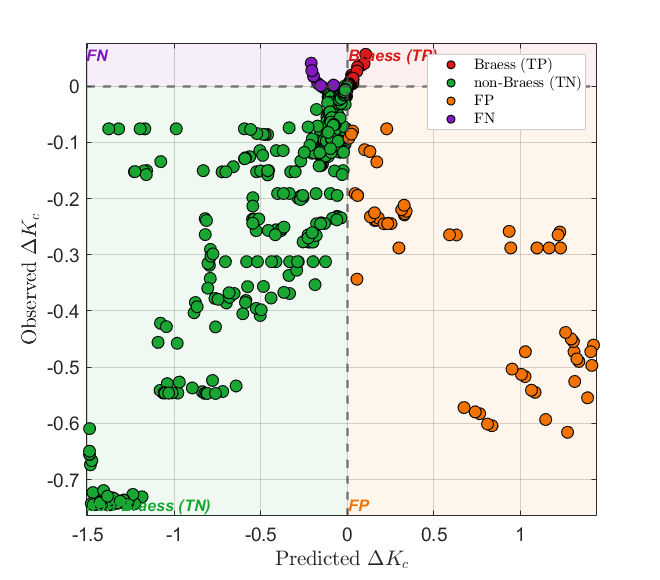}
    \end{minipage}

    \caption{(Left panel) Histogram of Braess predictor outcomes relative to the phase difference of coupled oscillators for twenty realizations of connected random graphs with $14$ nodes.  When restricted to edges coupling oscillators with absolute phase difference less than $\pi$, we find $506$ true negative cases, $29$ true positive cases, $7$ false negative cases, and $0$ false positive cases.  Thus, our predictor is accurate in over $98\%$ of these edges (phase difference less than $\pi$, relative deviation predicted to exceed $1\%$).  (Right panel) Scatter plot of observed deviations in $K_c$ versus the deviations predicted by Theorem~\ref{thm:saddlenode}.   }
    \label{fig:RGN1420}
\end{figure*}

Additional numerical investigations in support of $K_c'(0)$ being a useful predictor for the occurrence of Braess paradox are presented in Figures~\ref{fig:RGN30}-\ref{fig:ER2820}.  We present outcomes for twenty realizations each of random geometric graphs with $14$ and $30$ nodes and Erd\"os-R\'enyi random graphs with $14$ and $28$ nodes.  In each trial, we consider the addition of all edges for which the criterion in Theorem~\ref{thm:saddlenode} predicts a relative change in the critical coupling exceeding $0.1\%$, positive or negative.  Overall, we find that our predictor performs strongly whenever the phase difference between coupled oscillators is less than $\pi$ in magnitude.  For the random geometric graphs in particular, these results require filtering out edges whose predicted relative change falls below a more stringent $1\%$ threshold. To state some specific results, consider the example of random geometric graphs presented in Figure~\ref{fig:RGN1420}. When restricting our consideration to potential edges for which 
i) a relative change in $K_c$ of magnitude of $1\%$ or greater is predicted, and 
ii) coupling oscillators with an absolute phase difference less than $\pi$, we find $506$ true negative cases, $29$ true positive cases, $7$ false negative cases, and $0$ false positive cases.  
Thus, our predictor is correct in over $98\%$ of these cases.  Similar levels of success are obtained on larger networks and for Erd\"os-R\'enyi networks.

\begin{figure}[t]
    \centering
    \includegraphics[width=0.95\columnwidth]{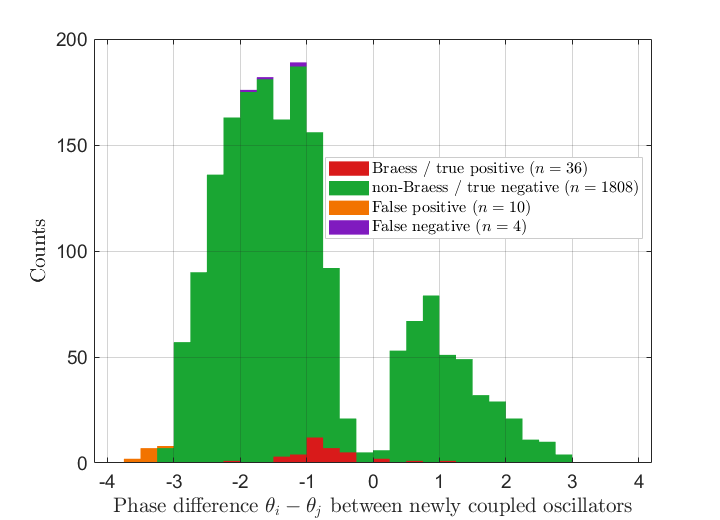}
    \caption{Histogram of Braess predictor outcomes relative to the phase difference of coupled oscillators for twenty realizations of connected random graphs with $30$ nodes, after restricting consideration to added edges with a predicted deviation in the critical coupling constant exceeding $1\%$.  }
    \label{fig:RGN30}
\end{figure}

\begin{figure*}[htbp]
    \centering


      \begin{minipage}[t]{0.42\textwidth}
        \centering
        \includegraphics[width=\textwidth]{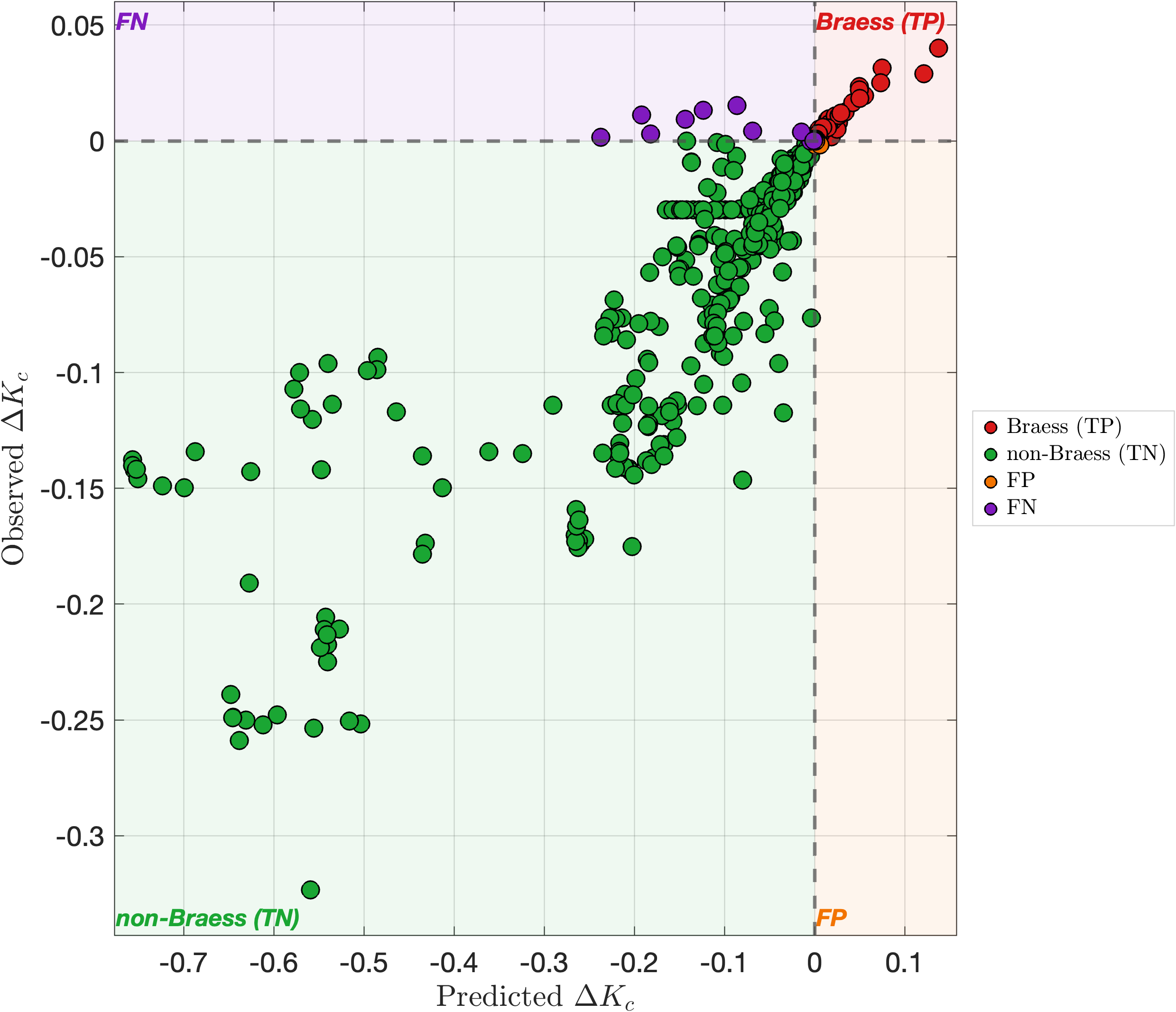}
    \end{minipage}
    \hspace{0.03\textwidth}
    \begin{minipage}[t]{0.405\textwidth}
        \centering
        \includegraphics[width=\textwidth]{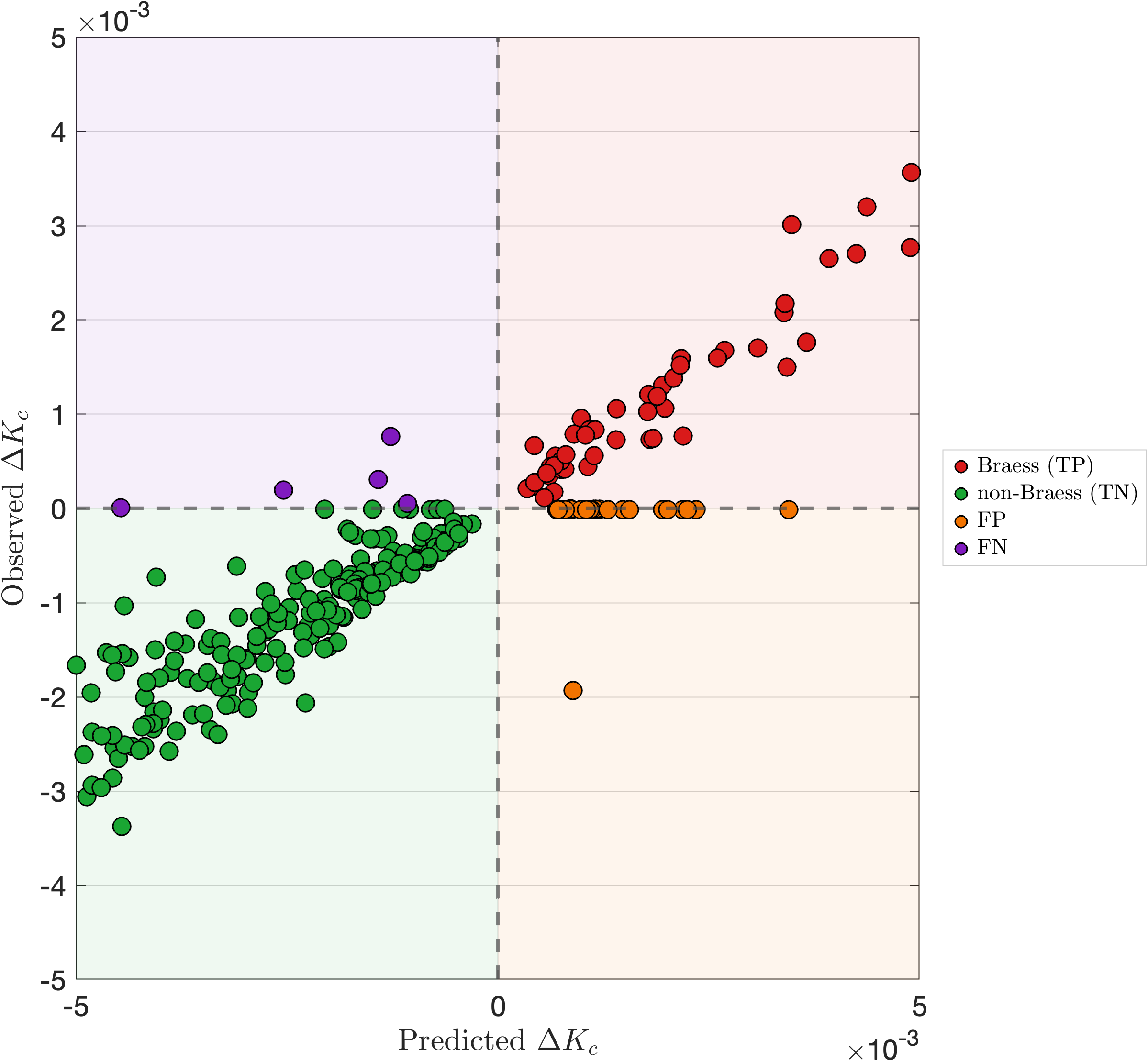}
    \end{minipage}

    \vspace{1em}

    \caption{(Left panel) Outcomes for twenty realizations of Erd\"os R\'enyi random networks with $14$ nodes.  Only edges with a predicted relative change in $K_c$ exceeding $0.1\%$ are included for consideration.  (Right panel) Same graph but zoomed in near the origin.        }
    \label{fig:ER1420}
\end{figure*}

\begin{figure*}[htbp]
    \centering
    \begin{minipage}[t]{0.37\textwidth}
        \centering
        \includegraphics[width=\textwidth]{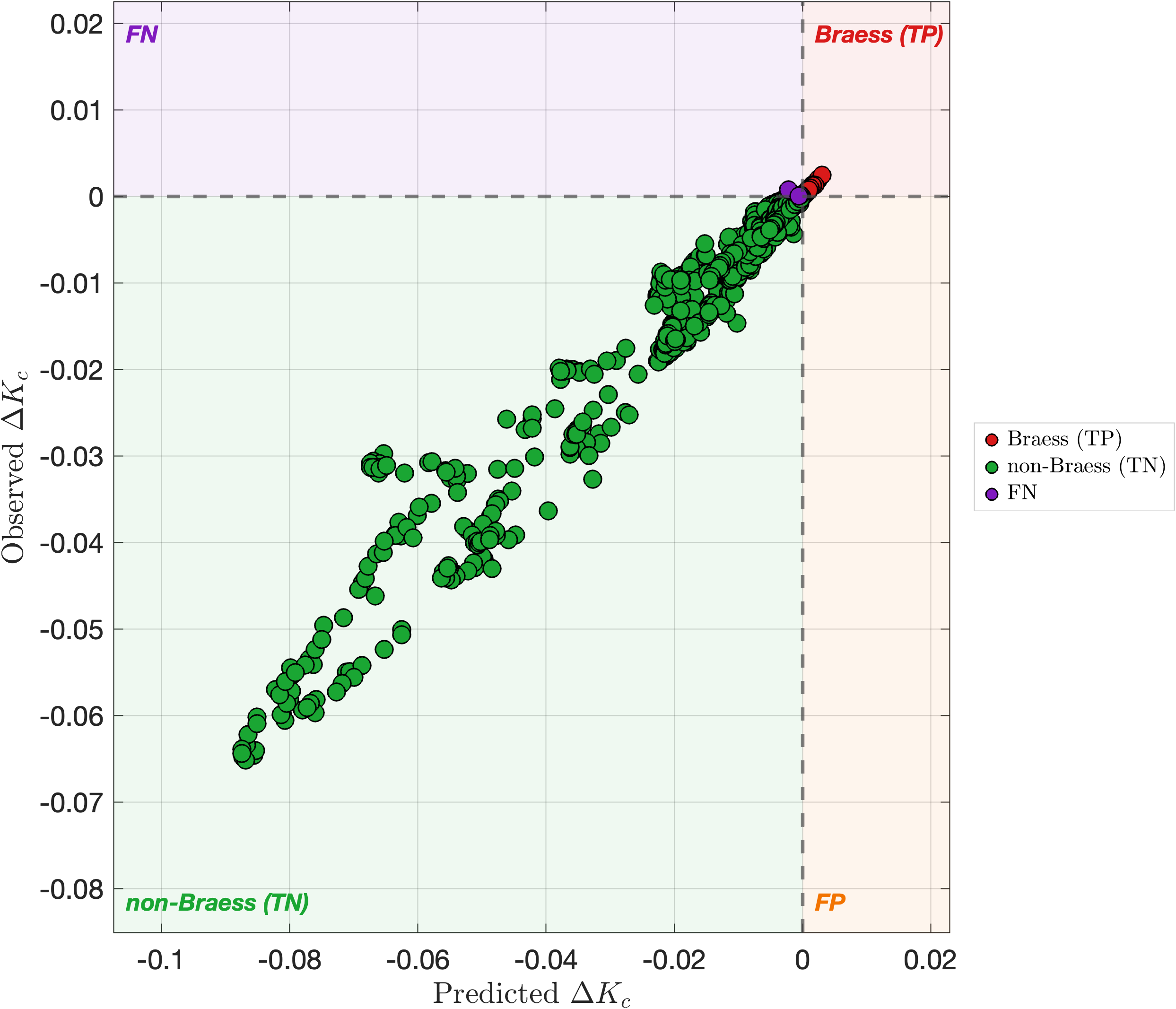}
    \end{minipage}
    \hspace{0.03\textwidth}
    \begin{minipage}[t]{0.42\textwidth}
        \centering
        \includegraphics[width=\textwidth]{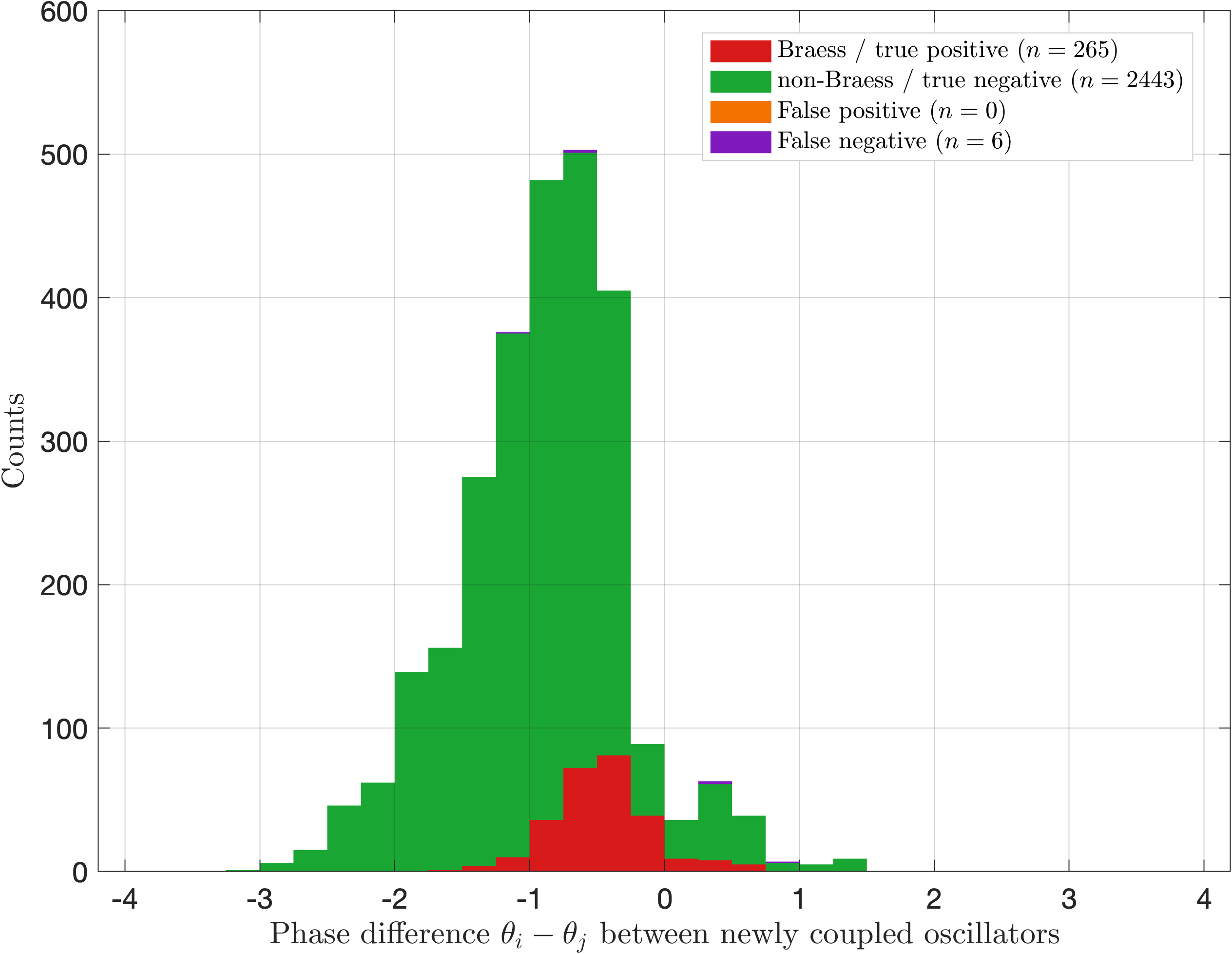}
    \end{minipage}

    \caption{(Left panel) Scatter plot of observed versus predicted changes in the critical coupling constant for twenty realizations of Erd\"os R\'enyi random networks with $28$ nodes.  Only edges with a predicted relative change in $K_c$ exceeding $0.1\%$ are included for consideration.  (Right panel) Histogram of the data presented in the left panel showing edge classification as a function of the phase difference between the coupled oscillators.  Note that the absolute change in $K_c$ for Braess edges is noticeably smaller for the larger network than the one with fewer nodes presented in Figure~\ref{fig:ER1420}; see the discussion of Erd\"os-R\'enyi graphs in Section~\ref{sec:motifs}.  }
    \label{fig:ER2820}
\end{figure*}

\paragraph{False positive and false negative edges}
Theorem~\ref{thm:saddlenode} provides some insight as to the mechanisms that might contribute to inaccuracies in our predictions.  Specifically, under generic assumptions, Theorem~\ref{thm:saddlenode} identifies two distinct possibilities by which $K_c(a)$ may not be monotone.  Suppose that $K_c'(0)>0$ indicating that $(p,q)$ is a locally Braess edge.  If this edge fails to be Braess, then necessarily, there must exist an $0<a^*<1$ for which $K_c'(a^*)=0$.  This can only occur if either 
i) the phases of the oscillators satisfy $\theta_p=\theta_q+k\pi$ for some $k\in \mathbb{Z}$, 
ii) the respective entries in the critical eigenvector match, i.e., $v_p=v_q$, or 
iii) a degeneracy occurs in which either the nontrivial critical eigenvector is no longer unique or remains unique, but the non-degeneracy condition $\bv^T\Omega\ne0$ fails.   
We observe that false positive cases where the phase difference between oscillators exceeds $\pi$ typically exhibit a case 
i) where increased coupling eventually draws the oscillators closer, and the derivative of $K_c'$ changes sign.  
We also note that most of the false negative cases that we observe arise due to 
ii) where the entries in the critical eigenvector coincide for some non-zero value of the coupling weight $a$.

\begin{figure*}[htbp]
    \centering

    \includegraphics[width=0.75\textwidth]{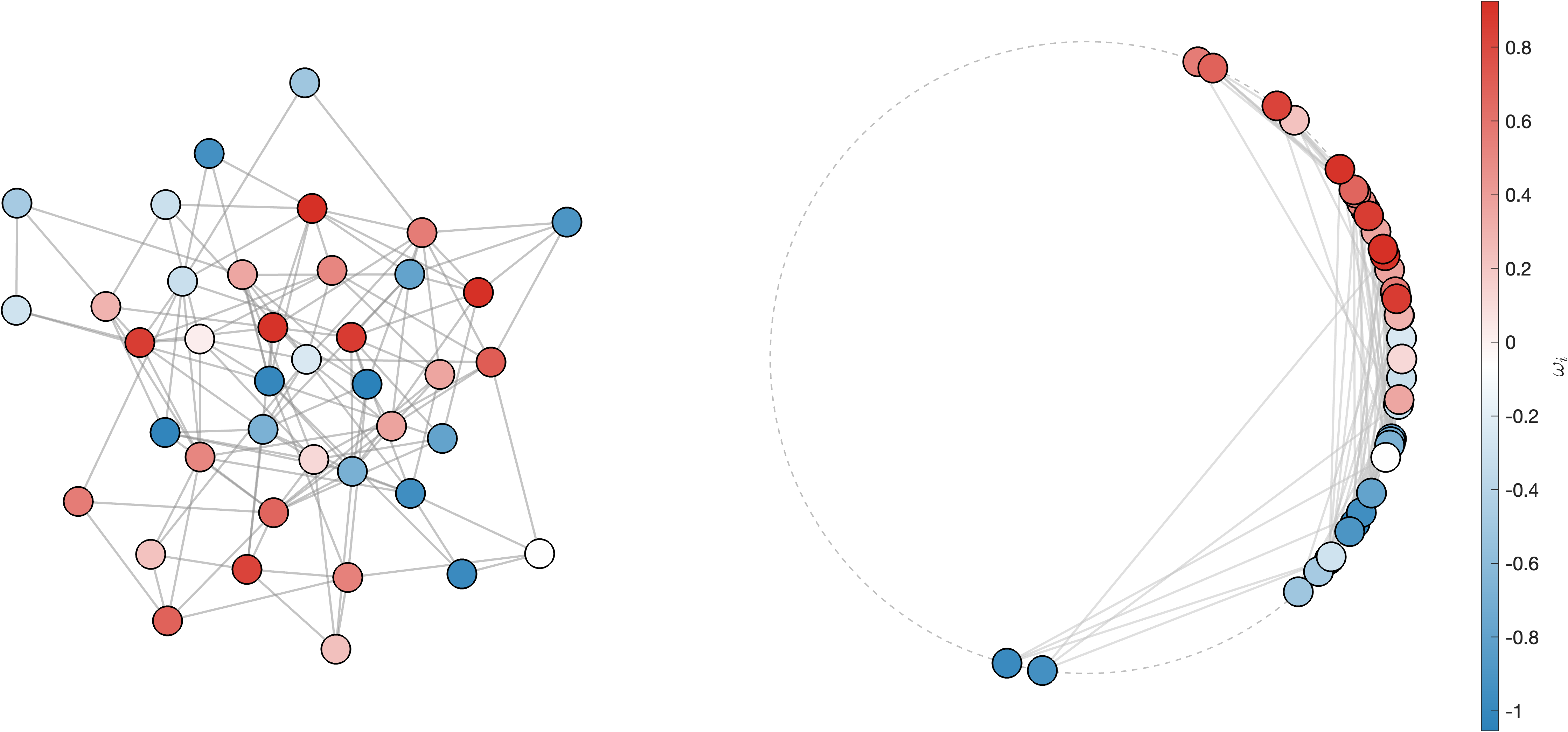}

    \vspace{1em}

    \begin{minipage}[t]{0.37\textwidth}
        \centering
        \includegraphics[width=\textwidth]{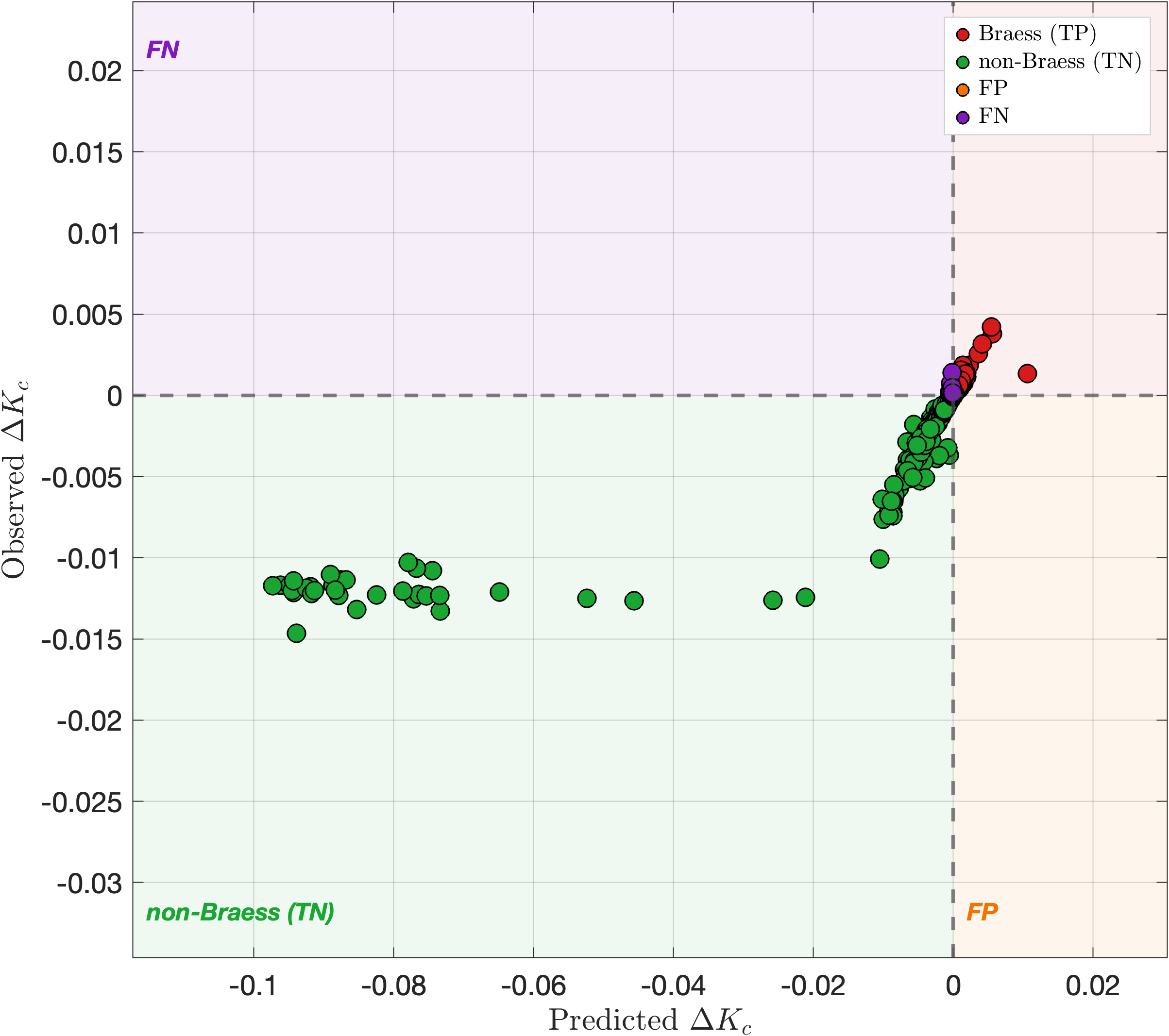}
    \end{minipage}
    \hspace{0.03\textwidth}
    \begin{minipage}[t]{0.42\textwidth}
        \centering
        \includegraphics[width=\textwidth]{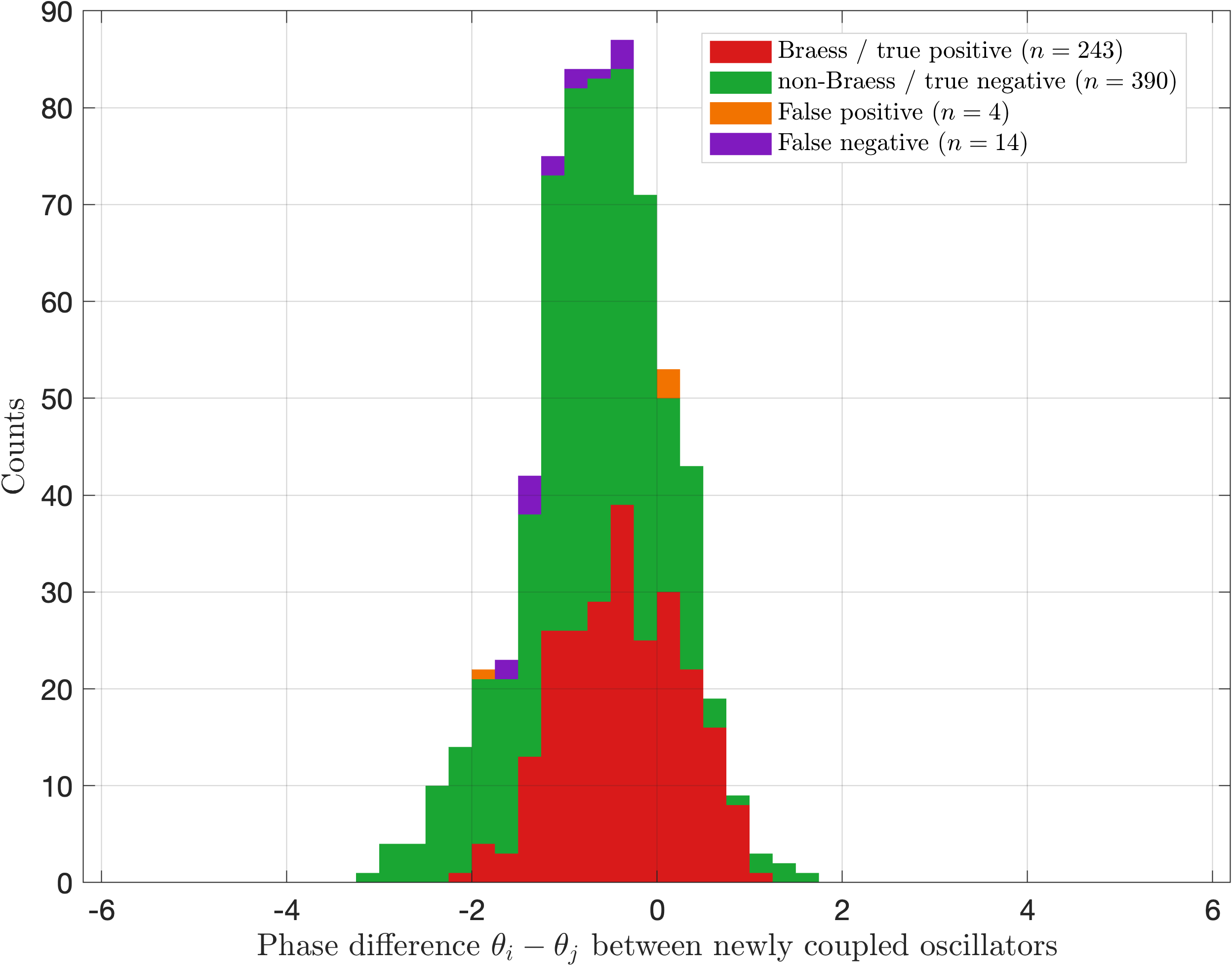}
    \end{minipage}
    \caption{Example Erd\"os-R\'enyi network and effect of edge addition on the critical coupling constant.  (Upper left panel) Network topology with nodes color-coded according to their natural frequency.  (Upper right panel) Network embedded in the unit circle with node location given by the phase at the critical value $K_c$.  (?Lower left panel) Scatter plot of observed vs predicted change in the critical coupling constant.  (Lower right panel) Histogram counting Braess and non-Braess edges as a function of the phase difference between the coupled oscillators.  Note that we do not filter edges with small $K'_c(0)$ values in this example.   }
    \label{fig:N401}
\end{figure*}

\paragraph{Representative examples.}
We provide four specific examples from our investigations.  Figure~\ref{fig:N401} and Figure~\ref{fig:N402} show Erd\"os R\'enyi random graphs with two qualitatively different outcomes.  Figure~\ref{fig:N401} represents a situation where our results work well.  The network has two differentiated nodes whose phases are distinct from the bulk of the network.  Added edges can either connect nodes within the bulk of the network or connect the differentiated nodes to nodes in the bulk.  The second case (bulk-to-isolated node connections) leads to large decreases in the critical coupling constant, although due to nonlinear effects, these decreases are limited in their total effect.  For nodes connecting two nodes in the bulk, our predictions work quite well with strongly linear behavior between predicted and observed deviations in the critical coupling.  Figure~\ref{fig:N402} shows a slightly different situation where there exists a single differentiated node that is connected to a single node in the bulk.  At the critical coupling constant, the critical eigenvector is piecewise constant, so additional edges that further couple the bulk to itself will have $K_c'(0)=0$ due to their critical eigenvector components being equal.  This situation is explained in more detail in the case of simple chain networks below.  

As we emphasized previously, our predictor is based upon the derivative $K_c'(0)$, and nonlinear effects could lead to false predictions. We show two examples of the graphs $K_c(a)$ showing how they may depend on the parameter $a$.  Figures~\ref{fig:true-positive-network}--\ref{fig:true-positive-kc-vs-a} 
show a representative true positive Braess edge in the case of a random geometric network.  
The highlighted added edge is predicted to be Braess by the sign of $K_c'(0)$,
and the full edge addition increases the numerically estimated threshold.
The sampled curve $K_c^{\mathrm{num}}(a)$ agrees with the initial direction
predicted by the derivative.  Figures~\ref{fig:false-positive-network}--\ref{fig:false-positive-kc-vs-a}
show a representative false positive example. The local derivative predicts
Braess behavior, but the full edge addition lowers the numerically estimated
critical coupling, illustrating the nonlinear limitation of the first-order
predictor.

\begin{figure}[t]
    \centering
    \includegraphics[width=0.95\columnwidth]{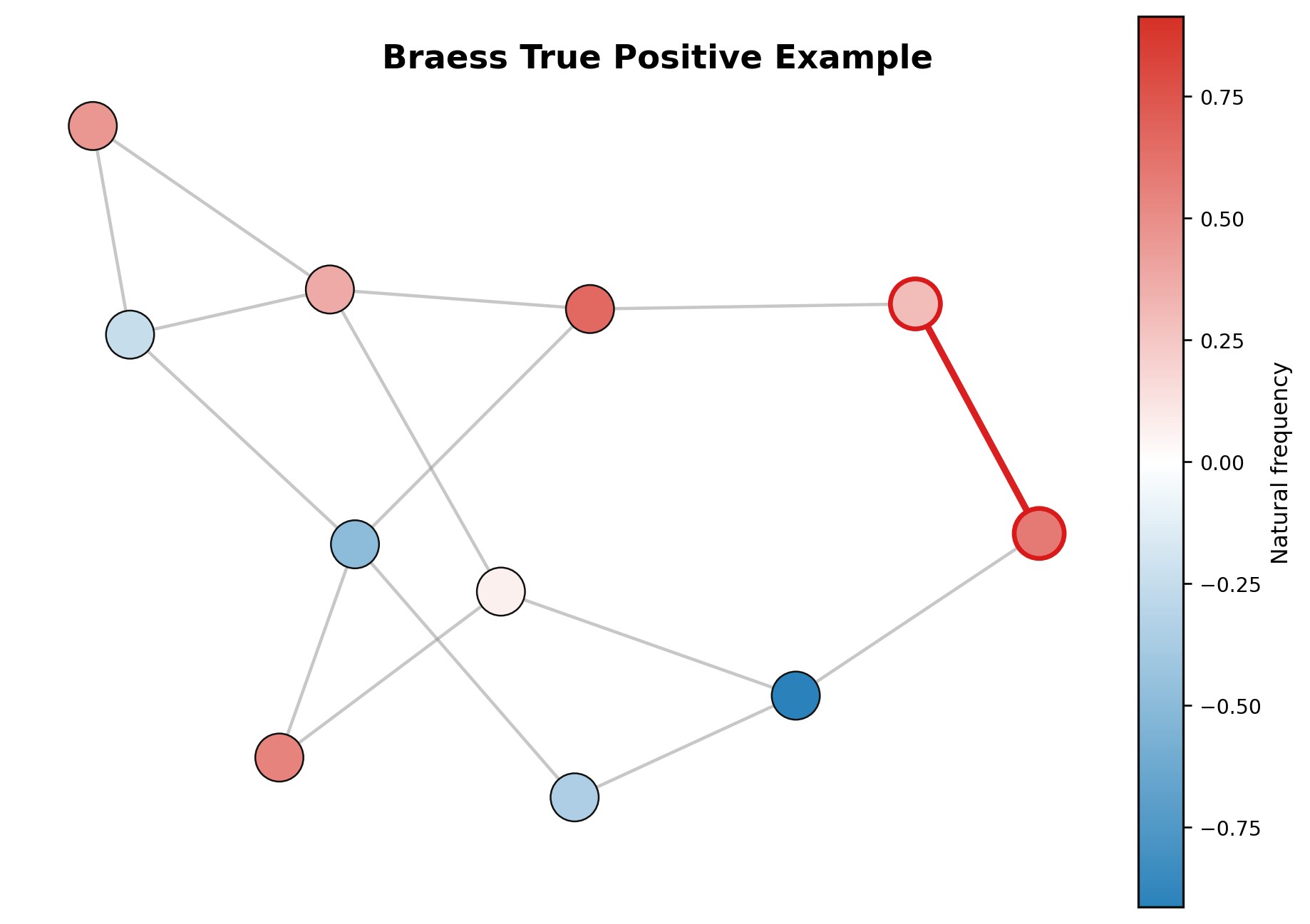}
    \caption{Random geometric network with a Braess edge identified (red edge).}
    \label{fig:true-positive-network}
\end{figure}

\begin{figure}[t]
    \centering
    \includegraphics[width=0.95\columnwidth]{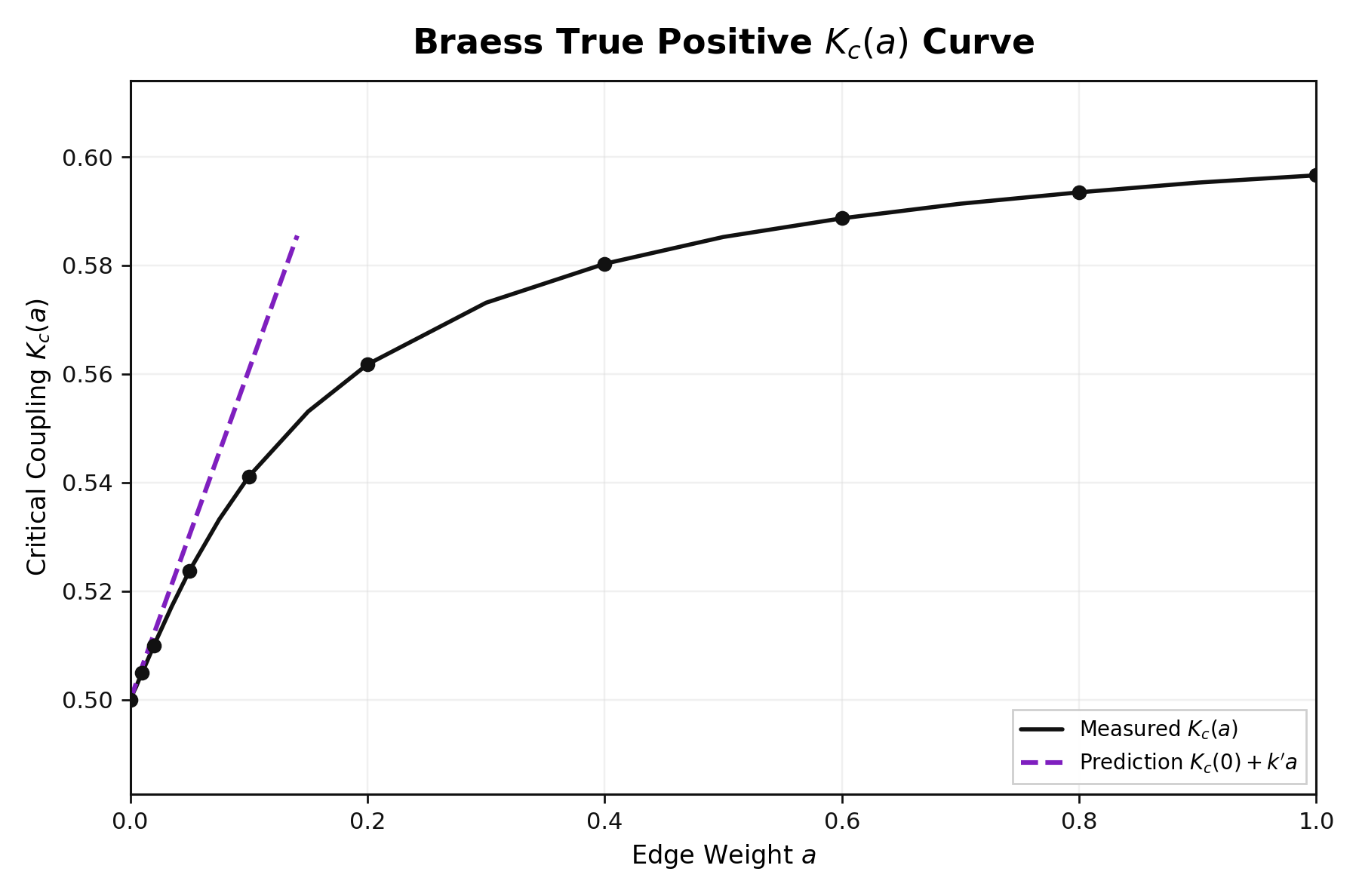}
    \caption{Critical coupling as a function of added edge weight $a$ associated to the red edge from Figure~\ref{fig:true-positive-network}.  The purple dashed line is the linear approximation guaranteed by Theorem~\ref{thm:saddlenode}.  The edge is therefore both a locally Braess edge ($K'_c(0)>0$ and a Braess edge $K_c(a)>K_c(0)$.  }
    \label{fig:true-positive-kc-vs-a}
\end{figure}

\begin{figure*}[htbp]
    \centering

    \begin{minipage}[t]{0.23\textwidth}
        \centering
        \includegraphics[width=\textwidth]{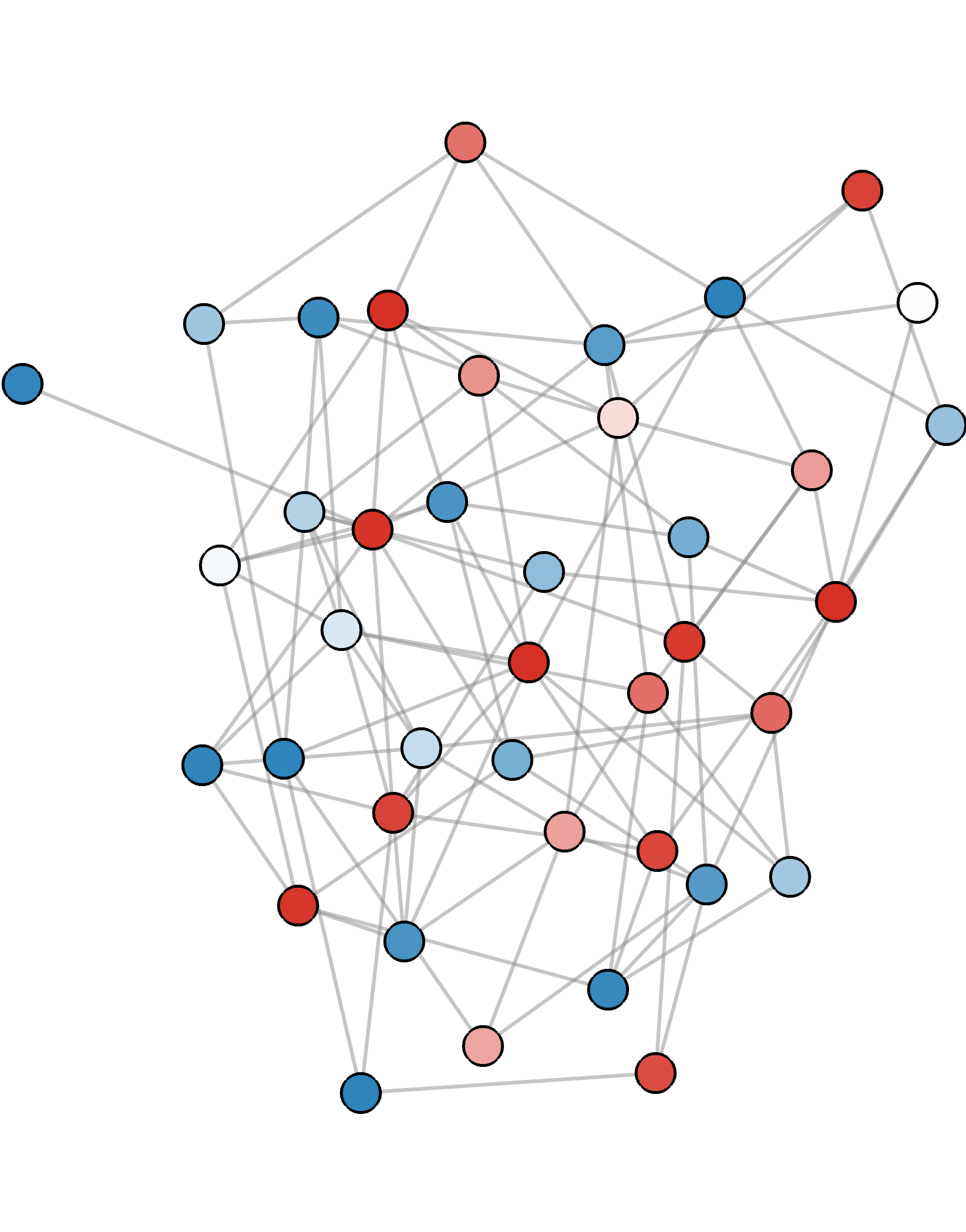}
    \end{minipage}
    \hfill
    \begin{minipage}[t]{0.27\textwidth}
        \centering
        \includegraphics[width=\textwidth]{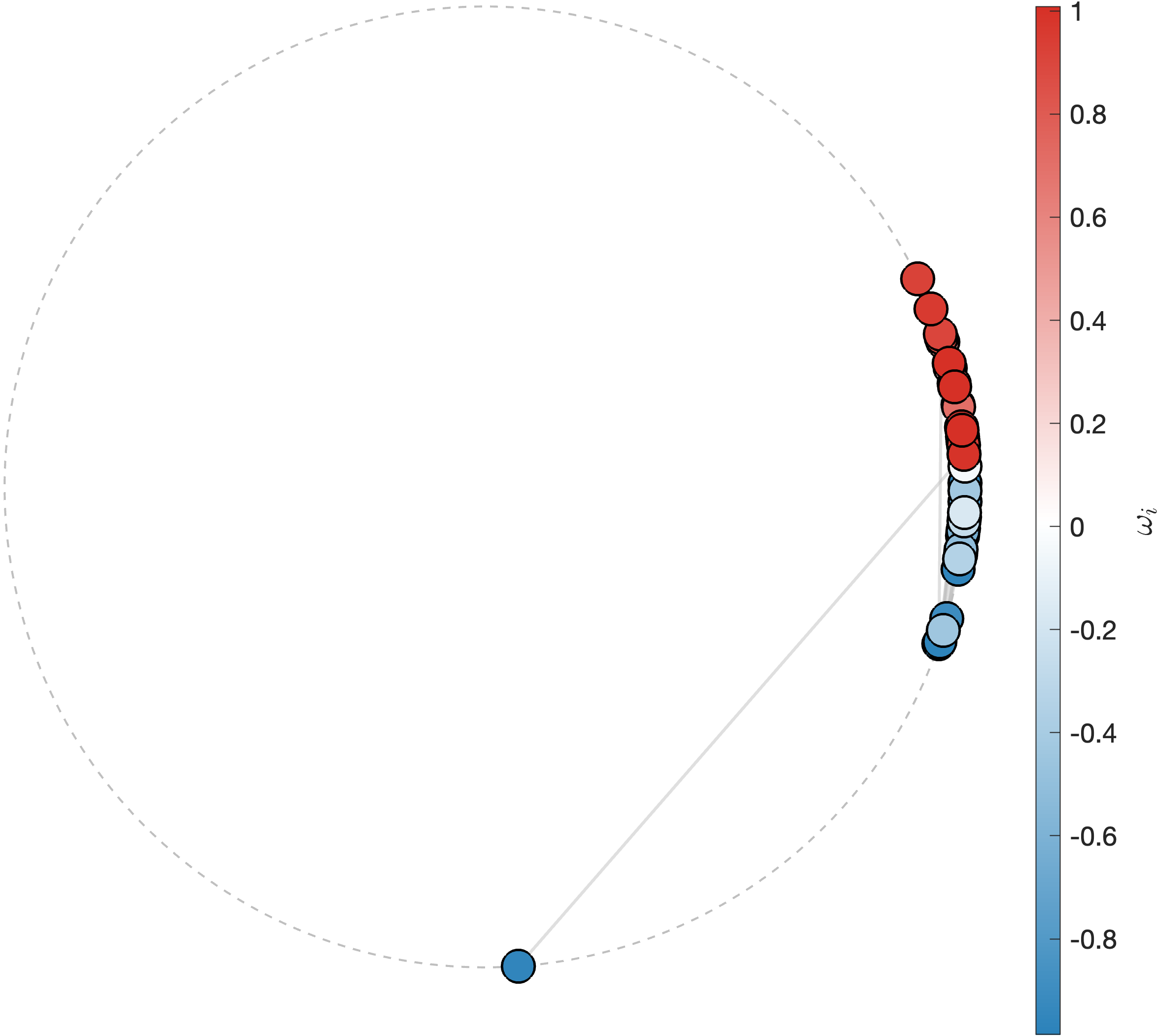}
    \end{minipage}
    \hfill
    \begin{minipage}[t]{0.36\textwidth}
        \centering
        \includegraphics[width=\textwidth]{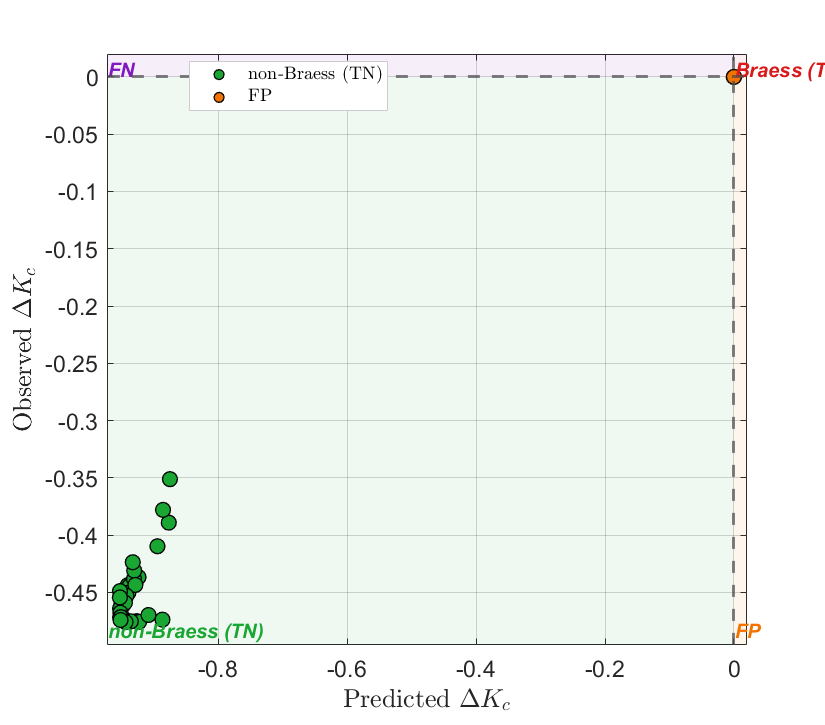}
    \end{minipage}

    \caption{Example Erd\"os-R\'enyi network and effect of edge addition on the critical coupling constant.  In contrast to Figure~\ref{fig:N401}, here the critical coupling value occurs due to a $\pi/2$ phase difference between the isolated node in the upper left of the network and its neighbor.  The situation here is exactly as in the chain example considered in Section~\ref{sec:motifs}.  The critical eigenvector is piecewise constant.  Adding edges between any two nodes lying in the bulk of the network has $K'_c(0)=0$ since the critical eigenvector components are equal. Thus, adding an edge between the isolated node and any node in the bulk will lead to a decrease in the critical coupling, as may be anticipated.  (Left panel) Network topology with nodes color-coded according to their natural frequency.  (Center panel) Network embedded in the unit circle with node location given by the phase at the critical value $K_c$.  (Right panel) Scatter plot of observed vs. predicted changes in the critical coupling constant for all possible edge additions.  There are a large number of edges where the local predictor is zero due to the piece-wise constant nature of the critical eigenvector, resulting in many data points near the origin in the scatter plot. }
    \label{fig:N402}
\end{figure*}

\begin{figure}[t]
    \centering
    \includegraphics[width=0.95\columnwidth]{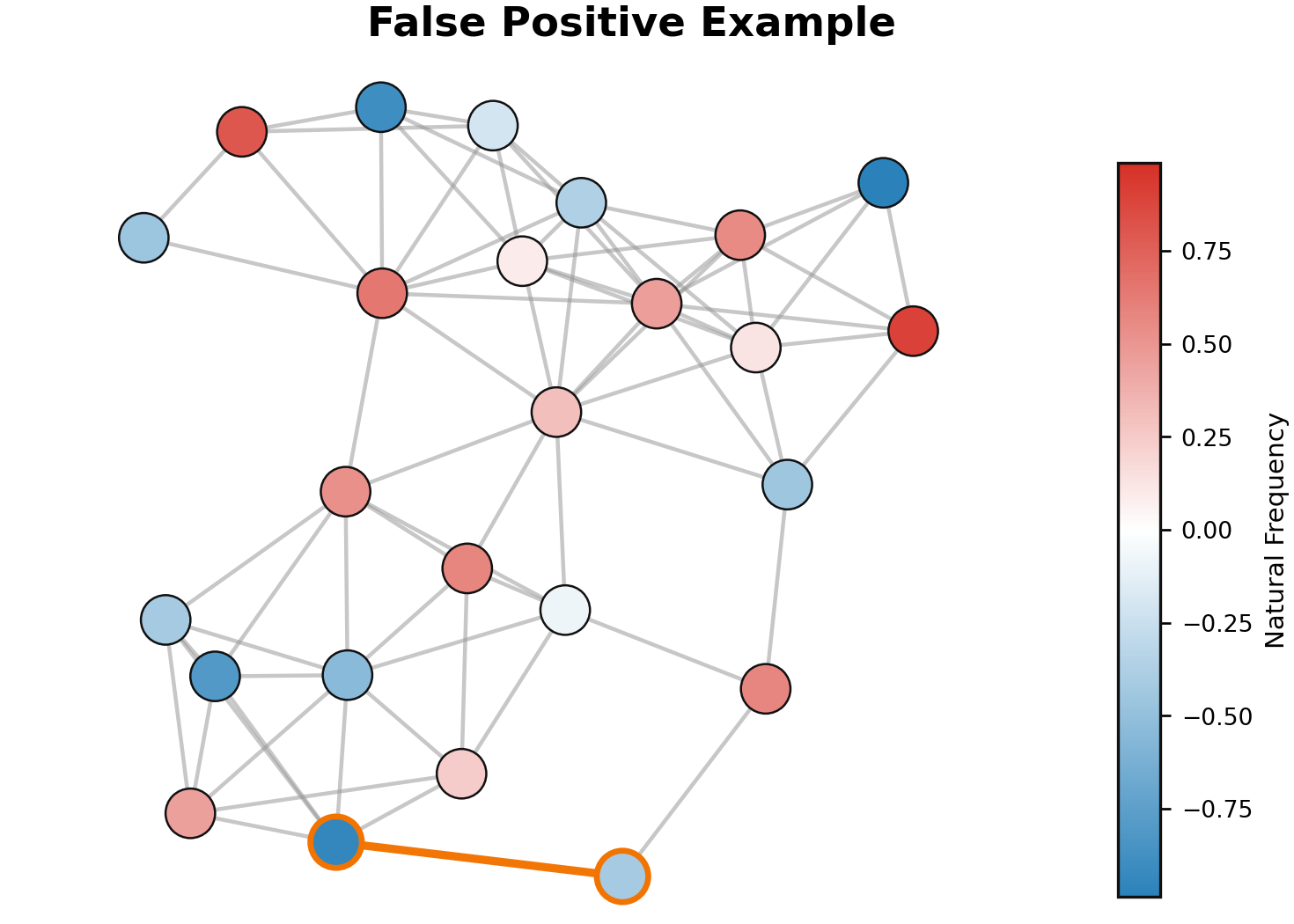}
    \caption{Random geometric network featuring an edge (orange edge) which is locally Braess but not Braess (false positive). }
    \label{fig:false-positive-network}
\end{figure}

\begin{figure}[t]
    \centering
    \includegraphics[width=0.95\columnwidth]{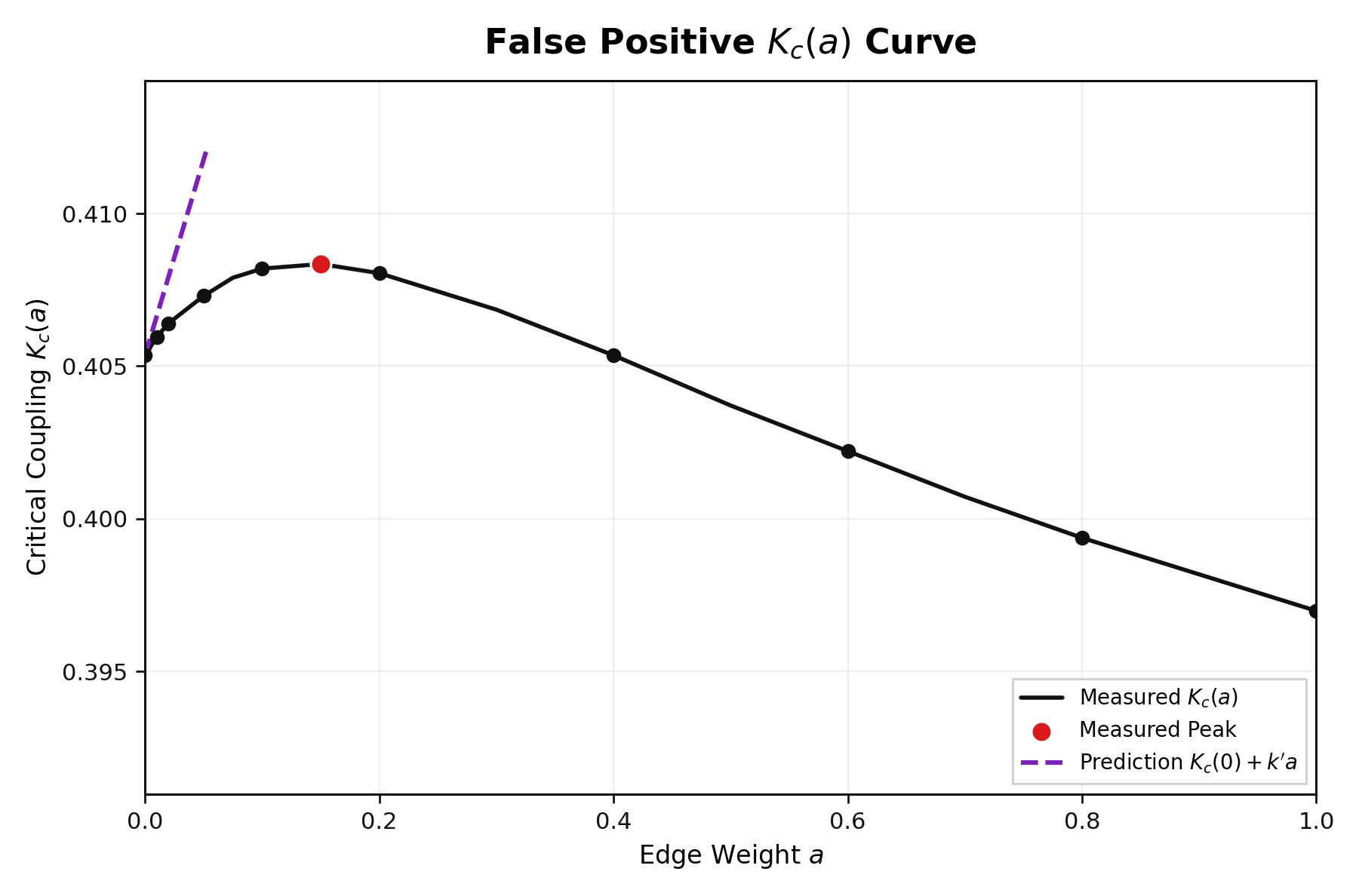}
    \caption{Critical coupling as a function of added edge weight $a$ for the representative false positive example from Figure~\ref{fig:false-positive-network}. The local derivative predicts Braess behavior, but the finite edge addition lowers the numerically estimated critical coupling by $a=1$.}
    \label{fig:false-positive-kc-vs-a}
\end{figure}

\section{Motifs and Examples }\label{sec:motifs}

In previous sections, we derived an explicit analytical criterion for the effect of a newly added edge on the critical coupling constant, and leveraged this local prediction to assess the impact of fully incorporating the edge into the network.  In this section, we consider several sample networks with simpler network topologies to illustrate various aspects of the theory.   We start with simple chain networks, where explicit computations are possible and to illustrate why we limit our classifier to only consider edges of coupling oscillators with phase differences less than $\pi$. We consider two network topologies from previous studies of Braess phenomena in coupled oscillator networks to compare and contrast 
our approach with prior work.  Finally, we briefly comment on the applicability of our results to large dense networks using some continuum limit analysis, making use of graphons.

\paragraph{Example: Chain Network}
The simplest way in which the linearization can gain a second zero eigenvalue is if there exists an edge in the graph such that if the edge is removed, the network becomes disconnected.  The most basic setting in which this arises is a chain network.  Slightly deviating from our notation thus far, we order nodes so that $(i,i+1)\in E$ for all $i=1,2,\dots,n-1$ and let $\omega_i$ be the natural frequency of oscillator $i$, so that the $\omega_i$ may no longer be increasing.  We still assume that $\Omega$ is mean zero, and we fix the phase condition $\theta_1=0$.  Let 
\[ \eta_k=\sum_{i=1}^{k-1} \omega_i.\]

Then there exists a steady state solution of the form 
\[ \theta_k=\theta_{k-1}+\mathrm{arcsin}\left(-\frac{\eta_k}{K}\right),\]
provided that $-K\leq \eta_k \leq K$ for all $k$.   We then obtain $K_c=\max\{|\eta_k|\}$.  When $K=K_c$ there exists an edge $(j,j+1)$ for which the phase difference $\theta_{j+1}-\theta_j=\pm\frac{\pi}{2}$. The Jacobian evaluated at this steady state is therefore block diagonal, and the critical eigenvector is piece-wise constant, taking the form:
\[ (v_c)_i=\left\{ \begin{array}{cc} \frac{1}{j} & i\leq j \\ -\frac{1}{n-j} & i>j \end{array}\right.. \]
Without loss of generality, we suppose $v_c^T\Omega>0$.  Now add the edge $(p,q)$ to the network, for some $(p,q)\neq (j,j+1)$. Theorem~\ref{thm:saddlenode} then implies that -- supposing that the phase difference between oscillators is less than $\pi$ in magnitude -- this edge will be locally Braess if and only if $p\leq j<q$, and $\theta_p<\theta_q$.

Now consider a 6 node chain with $\Omega=(-1,-1,-1,1,1,1)^T$, see Figure \ref{fig:chain6}.  Then $\eta=(0,-1,-2,-3,-2,-1)^T$ and we have $K_c=3$.  The critical eigenvector is a scalar multiple of the frequency vector: $\mathbf{v}_c=\frac{1}{\sqrt{6}}\Omega$.  Theorem~\ref{thm:saddlenode} then implies that $K_c'(0)\neq 0$ if and only if $p\leq 3<q$. Assuming then that $p\leq 3<q$ we find 
\[ K_c'(0)=-3\sqrt{6}\sin(\theta_q-\theta_p),\]
where we have used $K_c(0)=3$, $\bv_c^T\Omega=1$ and $v_p-v_q=-\frac{2}{\sqrt{6}}$.  We then obtain the following:  if the phase difference between the newly coupled oscillators is less than $\pi$, then $K_c'(0)<0$, the edge addition promotes synchrony, and is not locally Braess.  However, if the phase difference is larger (between $\pi$ and $2\pi$), then the edge is locally Braess and, at least for small $a$, the addition of the new edge will require a larger coupling constant to achieve synchrony.  

The reason for this is simple: when $a=0$, the critical coupling constant occurs because the phase difference between oscillators 3 and 4 equals $\pi/2$.  Coupling between oscillators 1 and 6 will draw these oscillators closer together, but because the phase differences are greater than $\pi$, this has a secondary effect of increasing phase differences between the remaining oscillators, contrary to what is required to establish synchrony; see Figure~\ref{fig:chain6}. Therefore, a larger coupling strength is required to maintain the phase-locked pattern. It is interesting that in many of the examples that we have considered, these locally Braess edges are not often Braess when the edge is fully incorporated in the network. 

\begin{figure}[htbp]
    \centering
    \includegraphics[width=0.45\textwidth]{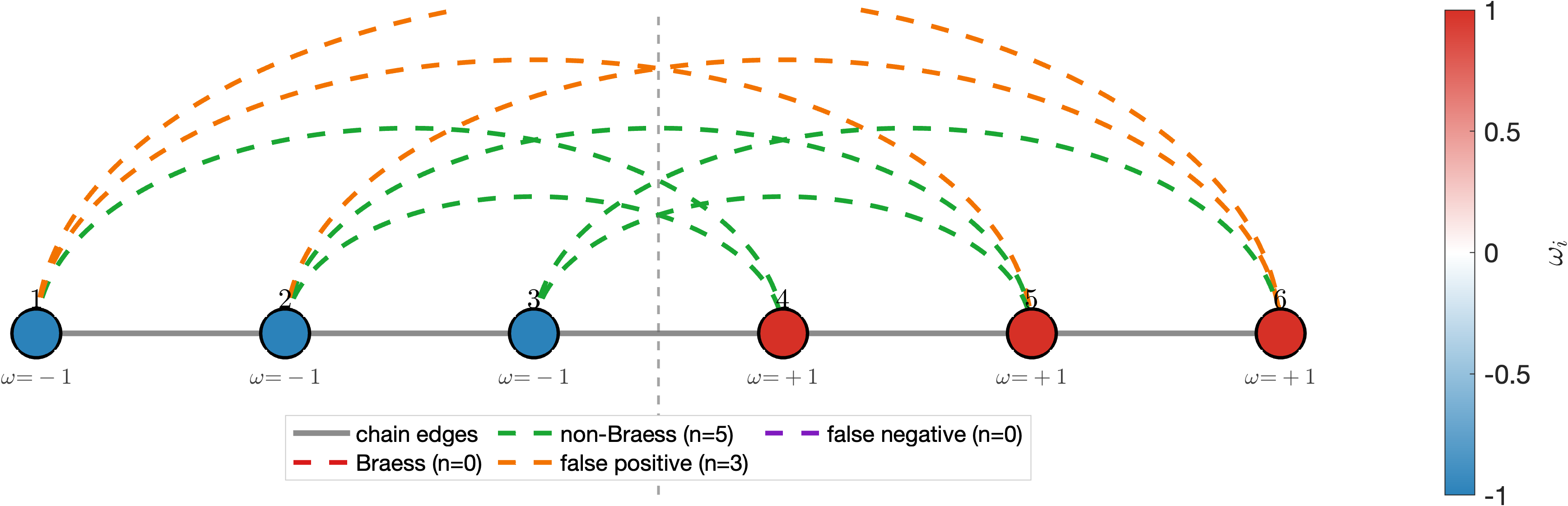}
    \vspace{1em}

    \begin{minipage}[t]{0.32\textwidth}
        \centering
        \includegraphics[width=\textwidth]{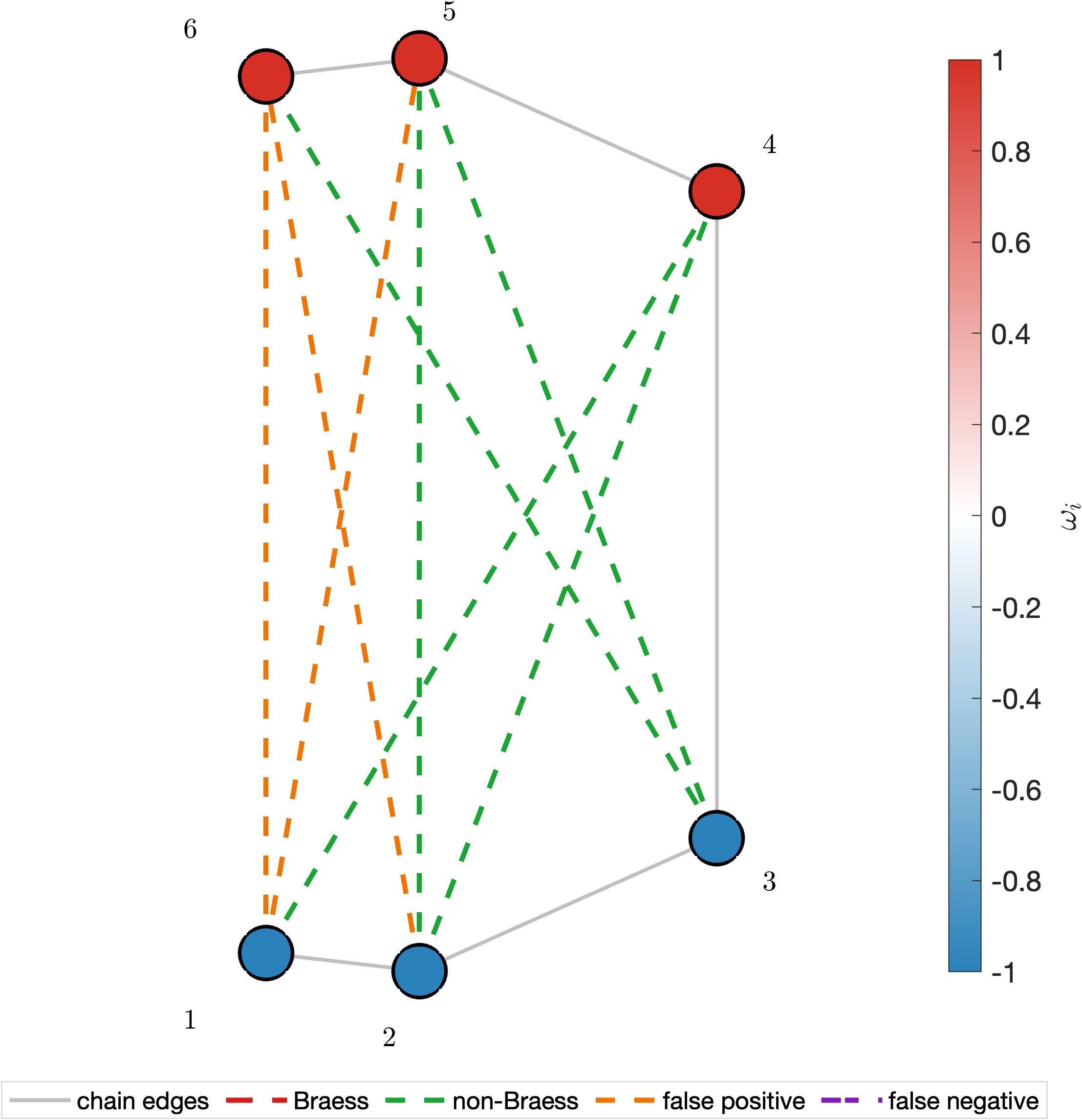}
    \end{minipage}
       \vspace{1em}
       
    \begin{minipage}[t]{0.42\textwidth}
        \centering
        \includegraphics[width=\textwidth]{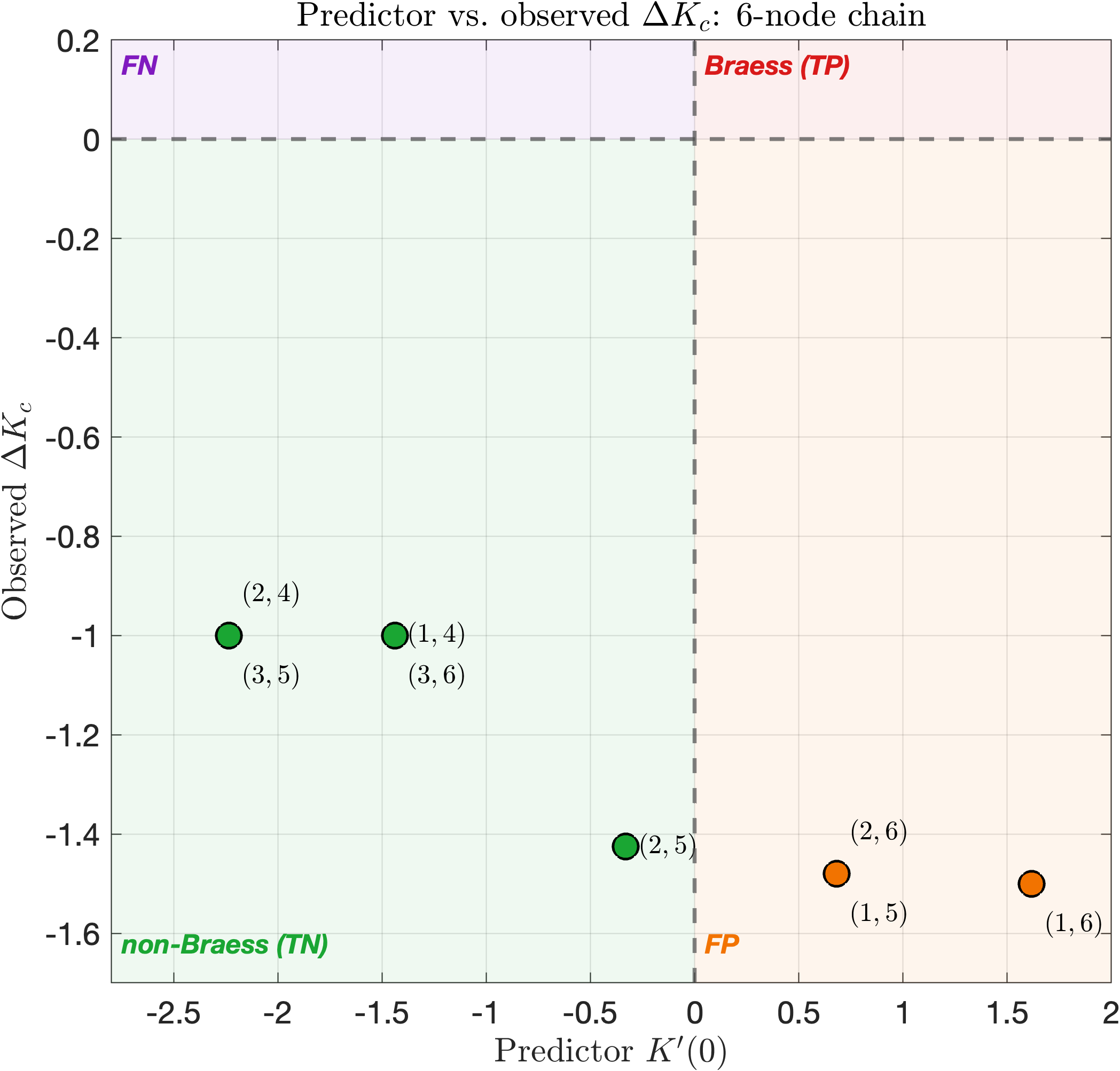}
    \end{minipage}

    \caption{%
    \textbf{6-node chain: Braess edge classification}
    ($\Omega = (-1,-1,-1,1,1,1)^T$, $K_c = 3$).
    \textit{Top:} All candidate non-edges drawn as dashed arcs;
    cross-partition edges ($p \leq 3 < q$) arc above the chain.
    \textit{Middle:} Nodes at their phase-locked positions
    $(\cos\theta_i, \sin\theta_i)$ at the saddle-node bifurcation.
    \textit{Bottom:} IFT predictor $K'_c(0)$ vs.\
    observed $\Delta K_c$; correctly classified edges lie in
    quadrants~I (Braess) and~III (non-Braess).  In accordance with our expectations, when the phase difference between oscillators exceeds $\pi$ in magnitude, then these edges are locally Braess but not Braess (false positive cases).  
    }
    \label{fig:chain6}
\end{figure}

Going beyond chain networks, a similar line of analysis may be applied to networks where the critical coupling is determined by the existence of an edge connecting two oscillators with a phase difference of $\pi/2$.  If removing this edge leads to a disconnected graph, then the analysis of the chain network can be immediately extended. To do so, consider the two subgraphs after the removal of this edge. Locally Braess edges are only possible for new edges spanning these two components.  We illustrate one such situation arising in an Erd\"os R\'enyi network in Figure~\ref{fig:N402}.

\paragraph{Example: Erd\"os R\'enyi random graphs in the large node limit}

Another perspective can be gained by passing to a continuum limit where the discrete adjacency matrix is replaced by a continuous function $W(x,y)$ -- referred to as a graphon\cite{borgs2008convergent,bramburger2025persistence,medvedev2019continuum,lovasz2012large} -- whereby the discrete system (\ref{eq:kuramoto}) is approximated by the nonlocal equation:
\begin{equation} \label{eq:graphon}
    u_t=\Omega(x)+K\int_0^1 W(x,y)\sin(u(y)-u(x))\mathrm{d}y,
\end{equation}
where $\Omega(x)$ is a continuous function describing the frequency of the oscillator at location $x$, and $W$ describes the strength of the connection between nodes at latent position $x$ and $y$.  We assume that $\Omega(x)$ is non-decreasing with $\Omega(0.5)=0$ and odd with respect to $x=0.5$ so that it has a mean of zero.   The existence of a phase-locked state $u^*(x)\in C[0,1]$ for various distribution functions $\Omega$ was studied by Ermentrout \cite{ermentrout1985synchronization} in the case of all-to-all coupling, where $W(x,y)=1$, while extensions to Erd\"os-R\'enyi random graphs with $W(x,y)=p<1$  were recently considered in Bramburger and Holzer\cite{bramburger2025capturing}.  For bimodal frequency distributions, the transition to synchrony in (\ref{eq:graphon}) occurs due to a saddle-node bifurcation at a critical coupling parameter $K_c$.  To study how changes in the graph impact the onset of synchrony, we consider the equation 
\begin{equation} \label{eq:graphon}
    u_t=\Omega(x)+K\int_0^1 \left[ W(x,y)+a\tilde{W}(x,y)\right]\sin(u(y)-u(x))\mathrm{d}y,
\end{equation}
where $\tilde{W}(x,y)\geq 0$ describes the addition of edges to the graph in accordance with the discrete setting in (\ref{eq:F}).  The following generalization of condition (\ref{eq:Kprime}) holds here:
\[ K_c'(0)=K_c(0)^2\frac{\int_0^1\int_0^1 v_c(x)\tilde{W}(x,y) \sin(u^*(y)-u^*(x))\mathrm{d}y\mathrm{d}x}{\int_0^1 \Omega(x)v_c(x)\mathrm{d}x}.\]
Suppose for the moment that $W(x,y)$ is  supported near the pair of points $(x_p,x_q)$ and $(x_q,x_p)$ with $x_p<x_q$.  Then, normalizing $\bv_c$ so that the denominator is positive, we find
\[ K'_c(0)=C\sin(u^*(x_q)-u^*(x_p))\left(\bv_c(x_p)-\bv_c(x_q)\right), \]
for some $C>0$.  Then, since $u^*(x)$ is monotone increasing with range contained in the interval $[-\pi/2,\pi/2]$, we find that a Braess phenomenon can occur only if there exist points $x_p,x_q$ such that $x_p<x_q$ and $\bv_c(x_p)>\bv_c(x_q)$.  In other words, Braess phenomenon is possible if the critical eigenfunction $\bv_c(x)$ is non-monotone.  While explicit representations of $\bv_c(x)$ are not often available, for all the distribution functions $\Omega(x)$ that we checked, the critical eigenfunction $\bv_c(x)$ is monotone.  This suggests that there is no systematic way to add edges to an Erd\"os-R\'enyi graph which will lead to shifts in $K_c$ that remain order one in the limit as the number of nodes in the random graph tends to infinity.  In fact, any graphon perturbation $\tilde{W}(x,y)>0$ will tend to generically decrease the critical coupling threshold, especially if it enhances connections between nodes near $x=0$ and $x=1$.

\paragraph{Example: simple ring network with symmetry}

When the network under consideration has some symmetry, it may be the case that the Jacobian has a kernel of dimension greater than two, and Theorem~\ref{thm:saddlenode} will not apply.  In these situations, Theorem~\ref{thm:stablepert} remains applicable.  We illustrate this in the following example, studied also in Section~IV.D (Figure~2) of Manik, Witthaut and Timme~\cite{manik2017cycle}.  We deviate somewhat from the previous examples and consider here changes in graph weights as in (\ref{eq:weightedkuramoto}) rather than new edge addition.  The example is interesting in that it provides an example of a network modification which both shifts the $K_c$ value higher\cite{manik2017cycle} while also increasing the measure of the order parameter according to Theorem~\ref{thm:stablepert}.

Consider a 4-node ring network $1$-$2$-$3$-$4$-$1$ where nodes 
$1$ and $3$ have natural frequency $-1$ and nodes $2$ and $4$ 
have natural frequency $+1$, so that 
$\boldsymbol{\Omega}=(-1,1,-1,1)^T$. The network has the following 
steady state solution (imposing the phase condition $\theta_1=0$):
\[ \boldsymbol{\theta}^*(K)=\left(0,\bar{\theta},0,
\bar{\theta}\right)^T, \]
where $\bar{\theta}=\arcsin\left(\frac{1}{2K}\right)$, valid for 
any $K\geq \frac{1}{2}$. The critical coupling constant is therefore $K_c=\frac{1}{2}$. 
The Jacobian $DF(\boldsymbol{\theta}^*,K)$ takes the form:
\[
DF(\boldsymbol{\theta}^*,K)=K\cos\bar{\theta}
\begin{pmatrix}
-2 & 1 & 0 & 1 \\
1 & -2 & 1 & 0 \\
0 & 1 & -2 & 1 \\
1 & 0 & 1 & -2
\end{pmatrix},
\]
since every edge of the ring crosses the same phase gap
$\bar{\theta}$, contributing $K\cos\bar{\theta}$ to each
off-diagonal entry.

At $K=K_c=\frac{1}{2}$, we have $\bar{\theta}=\frac{\pi}{2}$,
hence $\cos\bar{\theta}=0$, so the Jacobian vanishes 
identically. The kernel of $DF(\boldsymbol{\theta}^*,K_c)$ 
is therefore all of $\mathbb{R}^4$, which is four-dimensional 
rather than the two-dimensional kernel required by 
Theorem~\ref{thm:saddlenode}. Consequently 
Theorem~\ref{thm:saddlenode} does not apply at this 
bifurcation point, and we instead apply Theorem~\ref{thm:stablepert} for $K>\frac{1}{2}$, where 
the steady state is stable and all eigenvalues $\lambda_2,\lambda_3,\lambda_4$ are strictly negative.

The Jacobian is $-K\cos\bar{\theta}$ times the 4-cycle
Laplacian, whose eigenvalues are $\{0,2,2,4\}$, giving
the ordered eigenvalues
\[
\lambda_1=0, \quad \lambda_2=\lambda_3=-2K\cos\bar{\theta},
\quad \lambda_4=-4K\cos\bar{\theta},
\]
with corresponding orthonormal eigenvectors $v_1=\frac{1}{2}\bf{1}$ and 
\[
v_2=\frac{1}{2}\begin{pmatrix}1\\1\\-1\\-1\end{pmatrix}, \
v_3=\frac{1}{2}\begin{pmatrix}1\\-1\\-1\\1\end{pmatrix}, \
v_4=\frac{1}{2}\begin{pmatrix}1\\-1\\1\\-1\end{pmatrix}.
\]
Note that $\lambda_2 = \lambda_3$ is a degenerate eigenvalue, arising because every edge of the ring crosses the same phase gap $\bar{\theta}$, so the Jacobian has uniform off-diagonal entries. Any linear combination of $v_2$ and $v_3$ is also a valid eigenvector for this eigenvalue; the choice above gives a convenient orthonormal basis.

The base network is the 4-cycle $1$-$2$-$3$-$4$-$1$ with all entries of the adjacency matrix equal to $1$, so that each edge carries effective coupling strength $K$. We consider the effect of increasing the weight of the existing edge $(p,q)=(2,3)$ by perturbing its weight from $1$ to $1+a$, applying Theorem~\ref{thm:stablepert}. Since $\boldsymbol{\theta}^*=(0,\bar{\theta},0,\bar{\theta})^T$, we compute
\[
C = \sum_{j=1}^4\cos\theta_j^* = 1+\cos\bar{\theta}+1+\cos\bar{\theta}
= 2+2\cos\bar{\theta},
\]
\[
S = \sum_{j=1}^4\sin\theta_j^* = 0+\sin\bar{\theta}+0+\sin\bar{\theta}
= 2\sin\bar{\theta},
\]
and applying $w_j = S\cos\theta_j^* - C\sin\theta_j^*$ at each node:
\[
w_1 = 2\sin\bar{\theta}\cdot 1 - (2+2\cos\bar{\theta})\cdot 0
= 2\sin\bar{\theta},
\]
\begin{align*}
    w_2 = & 2\sin\bar{\theta}\cdot\cos\bar{\theta}
-(2+2\cos\bar{\theta})\cdot\sin\bar{\theta}\\ 
= & 2\sin\bar{\theta}\cos\bar{\theta}
-2\sin\bar{\theta}-2\sin\bar{\theta}\cos\bar{\theta}\\
= & -2\sin\bar{\theta},
\end{align*}
\[
w_3 = 2\sin\bar{\theta}\cdot 1 - (2+2\cos\bar{\theta})\cdot 0
= 2\sin\bar{\theta},
\]
\[
w_4 = 2\sin\bar{\theta}\cdot\cos\bar{\theta}
-(2+2\cos\bar{\theta})\cdot\sin\bar{\theta}
= -2\sin\bar{\theta},
\]
giving
\[\mathbf{w} = 2\sin\bar{\theta}\,(1,-1,1,-1)^T.\]

Computing the projections $\gamma_k = \mathbf{w}\cdot v_k$:
\[
\gamma_2 = 2\sin\bar{\theta}\cdot\frac{1}{2}(1-1-1+1) = 0,
\]
\[
\gamma_3 = 2\sin\bar{\theta}\cdot\frac{1}{2}(1+1-1-1) = 0,
\]
\[
\gamma_4 = 2\sin\bar{\theta}\cdot\frac{1}{2}(1+1+1+1)
= 4\sin\bar{\theta}.
\]

For the edge $(p,q)=(2,3)$, we compute $v_{k,2}-v_{k,3}$:
\[
v_{2,2}-v_{2,3}=\frac{1}{2}(1)-\frac{1}{2}(-1)=1,
\]
\[
v_{3,2}-v_{3,3}=\frac{1}{2}(-1)-\frac{1}{2}(-1)=0,
\]
\[
v_{4,2}-v_{4,3}=\frac{1}{2}(-1)-\frac{1}{2}(1)=-1.
\]

Since $\gamma_2=\gamma_3=0$, only $k=4$ contributes:
\[
\sum_{k=2}^4\gamma_k\frac{v_{k,2}-v_{k,3}}{\lambda_k} =
\frac{4\sin\bar{\theta}\cdot -1}{-4K\cos\bar{\theta}} =
\frac{\sin\bar{\theta}}{K\cos\bar{\theta}}.
\]

Substituting into \eqref{eq:ravaries} with $n=4$ and
$\sin(\theta_3^*-\theta_2^*)=\sin(0-\bar{\theta})=-\sin\bar{\theta}$:
\[
\frac{d}{da} \left.\frac{r(\theta^*(a))^2}{2}\right|_{a=0} = -\frac{K}{16}\cdot-\sin\bar{\theta}\cdot
\frac{\sin\bar{\theta}}{K\cos\bar{\theta}} =
\frac{\sin^2\bar{\theta}}{16\cos\bar{\theta}}.
\]
Since $K>\frac{1}{2}$ implies $\bar{\theta}\in\left(0,\frac{\pi}{2}\right)$,
we have $\cos\bar{\theta}>0$ and $\sin\bar{\theta}>0$, so
$\frac{d}{da} r(\theta^*(a))>0$. This means that, for fixed $K>\frac{1}{2}$, the order parameter is locally increasing with respect to the edge weight perturbation $a$ at $a=0$; equivalently, increasing the weight of the edge $(2,3)$ from its nominal value increases the degree of synchronization.

\begin{remark}
We note that the two candidate new edges, the diagonals $(1,3)$ and $(2,4)$, both connect nodes having equal equilibrium phases, so $\sin(\theta_p^* - \theta_q^*) = 0$ for both, giving $\frac{d}{da} r(\theta^*(a))|_{a=0} = 0$ trivially. 
\end{remark}

In Manik {\em et al.}\cite{manik2017cycle} it is observed that the perturbation to the weights leads to an increase in the critical coupling constant and therefore the edge $(p,q)=(2,3)$ is locally Braess.  However, our application of Theorem~\ref{thm:stablepert} shows that, for any $K>K_c$, the order parameter also increases describing an increase in synchronization as measured by the order parameter $r(\theta)$; see Figure~\ref{fig:cycle}.   The two criteria thus operate on different aspects of the synchronization landscape.

\begin{figure}[htbp]
     \centering
     \includegraphics[width=0.45\textwidth]{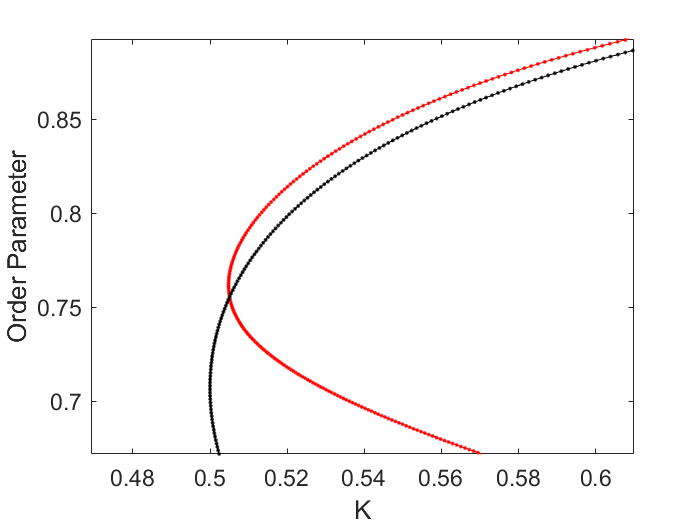}
     \caption{Order parameter as a function of the coupling constant $K$ for the cycle graph example presented in Section~\ref{sec:motifs}c.  Black depicts the order parameter associated to phase locked states with equal edge weights.  The critical coupling value is the minimum $K$ along each branch where the state undergoes a saddle-node bifurcation.   The red curve depicts the order parameter associated to phase locked states when the weight on edge $(p,q)=(2,3)$ is increased by $10\%$. Since $K_c$ shifts to a larger value we would classify this edge as locally Braess (although Theorem~\ref{thm:saddlenode} does not apply due to symmetry in the system) yet the overall order parameter simultaneously increases along the stable branch of solutions.  }
     \label{fig:cycle}
 \end{figure}

\section{Discussion}\label{sec:discussion}
In this work, we have studied the Braess paradox in Kuramoto oscillator 
networks through the lens of the Implicit Function Theorem. Our main result, Theorem~\ref{thm:saddlenode}, provides an analytical formula for the derivative $K_c'(0)$ that predicts whether adding a new edge to a network will raise or lower the critical coupling constant $K_c$, and thus whether the addition hinders or facilitates synchronization. The formula \eqref{eq:Kprime} is computationally inexpensive to evaluate: it requires only the phase-locked state $\boldsymbol{\theta}$, the critical eigenvector $\mathbf{v}$, and the natural frequencies $\boldsymbol{\Omega}$ of the 
base network, all of which are computed once and reused to screen every 
candidate edge. This stands in contrast to brute-force approaches that 
require a full numerical continuation for each candidate edge, which is 
computationally prohibitive for large networks.

The formula in Theorem~\ref{thm:saddlenode} admits a transparent 
geometric interpretation. The sign of $K_c'(0)$ is determined by the 
interplay between two quantities: the phase difference 
$\theta_q - \theta_p$ between the two oscillators to be connected, and 
the difference $v_p - v_q$ in the corresponding entries of the critical 
eigenvector. Specifically, the edge $(p,q)$ is predicted to be Braess 
when the phase ordering of the oscillators is \emph{opposite} to the 
ordering of their critical eigenvector components, provided the phase 
difference satisfies $|\theta_q - \theta_p| < \pi$. This dichotomy 
offers an intuitive explanation for the Braess phenomenon in oscillator 
networks: adding an edge between oscillators whose relative phases are 
misaligned with the critical mode of the network can inadvertently 
stiffen the system and raise the synchronization threshold. 

This geometric picture also clarifies when the predictor is expected to 
fail. When the phase difference $|\theta_q - \theta_p|$ exceeds $\pi$, then the addition of coupling between these two oscillators will have the downstream effect of increasing phase differences between the remaining oscillators, contrary to what is required for synchronization.  This leads to a larger coupling required for small coupling strengths, but interestingly, it seems to be the case that larger coupling strengths realign the oscillator phases so that the phase difference between oscillators spanning the coupled edge drops below $\pi$.  

Our numerical investigations, presented in Section~\ref{sec:numerics}, reveal that the predictor performs well 
across a range of network types, correctly classifying the majority of 
candidate edges as Braess or non-Braess. Two distinct failure modes were identified. \emph{False positives}, 
edges predicted to be Braess that are observed to be non-Braess, arise most frequently when the phase difference between the two oscillators is near $\pi$.  \emph{False negatives}, edges predicted to be non-Braess that are observed to be Braess, are rarer and tend to arise when the function $K_c(a)$ is non-monotone and the initial slope $K_c^{'}(0)<0$ underestimates the destabilizing effect of the new edge at finite weight $a$.

The $K_c(a)$ plots presented in Section~\ref{sec:numerics} further 
illuminate the limitations of the local prediction. While $K_c'(0)$ 
correctly predicts the initial direction of change for most edges, the 
function $K_c(a)$ can be non-monotone on $[0,1]$: an edge that initially 
raises $K_c$ may ultimately lower it, or vice versa. The chain 
example in Section~\ref{sec:motifs} provides a concrete illustration of 
this non-monotone behavior, where the interplay between local phase 
differences and global network topology produce a non-trivial dependence 
of $K_c$ on the edge weight $a$.

The motif analysis of Section~\ref{sec:motifs} reveals that the Braess phenomenon in oscillator networks is intimately connected to the 
structure of the critical eigenvector and the geometry of the 
phase-locked state. In chain networks, Braess edges arise when a new 
connection bridges two oscillators whose phase difference is large, 
effectively introducing a competing pathway that disrupts the phase 
coherence of the existing synchronous state. In ring networks with 
alternating frequencies, the high degree of symmetry leads to a 
degenerate critical point at which our theorem's hypotheses are not 
satisfied for the full nonlinear system, yet Theorem~\ref{thm:stablepert} 
can still be applied to characterize the behavior of the order parameter 
near stable synchronous states away from criticality.

The graphon perspective introduced in Section~\ref{sec:motifs} points 
toward a continuum limit in which our discrete formula may be replaced 
by an integral operator. This 
limit may allow for analytical predictions of Braess behavior in 
large random networks without requiring explicit computation of the 
critical eigenvector, and represents a promising direction for future 
work.   For dense Erd\"os R\'enyi networks, we presented an argument that there is no systematic way to add edges that are Braess when the number of nodes is taken to be sufficiently large.  We anticipate that other random network models based upon graphons should exist, such that they admit Braess edges in general.  One avenue to construct such graphs would be to consider stochastic block model graphons based upon lower-dimensional graphs that exhibit Braess edges.

Of course, several natural extensions of this work suggest themselves. First, the 
extension to weighted graphs described by~(\ref{eq:weightedkuramoto}) is 
immediate, as noted in Section~\ref{sec:theorem}, and raises the question 
of optimal edge weight allocation: given a budget of total edge weight to 
distribute across a set of candidate edges, how should the weights be 
chosen to minimize $K_c$? Second, the non-monotone behavior of $K_c(a)$ 
observed in Figure \ref{fig:false-positive-kc-vs-a} suggests that higher-order corrections 
to the IFT formula may provide a more accurate global predictor, 
particularly for edges near the decision boundary. Finally, the extension of our IFT-based classifier to directed and time-varying networks remains an open problem. In directed networks, the Jacobian $D_{\boldsymbol{\theta}}F$ is no longer symmetric, so the left and right kernels differ and the formula \eqref{eq:Kprime} must be modified accordingly. In time-varying networks, the notion of a phase-locked state itself requires reinterpretation. These extensions would broaden the applicability of the Braess paradox classifier to a wider class of real-world networks where directionality and non-stationarity are unavoidable features of the underlying dynamics.

\section*{Author Contributions} 
Author names are listed alphabetically.  
\textbf{Justin Arches:} Writing -- review \& editing.
\textbf{Emmanuel Fleurantin:} Conceptualization, Formal analysis, 
Methodology, Supervision, Writing -- original draft, 
Writing -- review \& editing.
\textbf{Matt Holzer:} Conceptualization, Formal analysis, Methodology, 
Project administration, Supervision, Writing -- original draft, 
Writing -- review \& editing.
\textbf{Siddhant Sood:} Software, Numerical Validation, Visualization, Writing -- review \& editing.
\textbf{Shakeeb Uddin:} Validation, 
Writing -- review \& editing.

\section*{Data Availability Statement}

The data that supports the findings of this study are available from the corresponding author upon reasonable request.

\section*{Acknowledgments}
M.H. was partially supported by the NSF grant DMS-2406623. The authors would like to thank the Mason Experimental Geometry Lab (MEGL) for providing a collaborative research environment that helped shape the ideas in this work. A special thanks goes out to Abdullah Hatif for his valuable contributions during the Fall 2025 semester, which helped lay the groundwork for this project.  Numerical continuation code written by Jason Bramburger in a prior collaboration\cite{bramburger2025capturing} was adapted to perform some of the numerical simulations conducted.  

\section*{Author Declarations}
The authors have no conflicts to disclose.
\begin{appendix}
    \section{Proof of Theorem~\ref{thm:saddlenode}}\label{appendixA}
Recall that throughout this section (and the manuscript), we impose $\displaystyle\sum_{i=1}^n \omega_i = 0$. This is
necessary for a phase-locked (steady-state) solution to exist, since summing
the steady-state equations over all $i$ gives $\sum_i \omega_i = 0$.

\begin{remark}
From the steady state map in Definition \ref{def:steadystate}, when $a = 0$, the map $F$ reduces to
\[
  F_i(\btheta, K, 0)
  = \omega_i + K \sum_{j=1}^n A_{ij} \sin(\theta_j - \theta_i),
\]
which is the original Kuramoto equation. For general $a$, the
$p$- and $q$-components of $F$ acquire the additional terms
\[
  F_p(\btheta, K, a)
  = \omega_p + K \!\sum_{j=1}^n A_{pj}\sin(\theta_j-\theta_p)
    + Ka\sin(\theta_q - \theta_p),
\]
\[
  F_q(\btheta, K, a)
  = \omega_q + K\! \sum_{j=1}^n A_{qj}\sin(\theta_j-\theta_q)
    + Ka\sin(\theta_p - \theta_q),
\]
while all other components are unchanged.
\end{remark}

Now fix $a = 0$. The Jacobian of $F$ with respect to $\btheta$ is

\begin{equation}
\label{eq:jacobian}
(D_{\btheta} F)_{ij} =
\begin{cases}
K A_{ij} \cos(\theta_j - \theta_i), & i \neq j, \\[4pt]
-\displaystyle\sum_{k \neq i} K A_{ik} \cos(\theta_k - \theta_i), & i = j.
\end{cases}
\end{equation}

\begin{lemma}[Kernel of the Jacobian]
\label{lem:kernel} Let $(\btheta_c, K_c)$ be a critical steady state, i.e.,
$F(\btheta_c, K_c, 0) = \bzero$ and $K_c$ is the smallest coupling for which
a steady state exists. Then $D_{\btheta}F(\btheta_c, K_c, 0)$ is a
weighted graph Laplacian, every row sums to zero, and
\[
  \dim \ker\!\bigl(D_{\btheta}F(\btheta_c, K_c, 0)\bigr) \geq 1.
\]
Moreover, there exists a nonzero vector $\bv$ such that
$D_{\btheta}F(\btheta_c, K_c, 0)\,\bv = \bzero$.
\end{lemma}

\begin{proof}
    For $i \neq j$, the only term in $F_i$ that depends on $\theta_j$ is
$K A_{ij}\sin(\theta_j - \theta_i)$, so
\[
  \frac{\partial F_i}{\partial \theta_j}
  = K A_{ij} \cos(\theta_j - \theta_i), \qquad i \neq j.
\]
For the diagonal, every off-diagonal term $KA_{ij}\sin(\theta_j - \theta_i)$
depends on $\theta_i$, giving
\[
  \frac{\partial F_i}{\partial \theta_i}
  = -K \sum_{j \neq i} A_{ij} \cos(\theta_j - \theta_i).
\]
This yields the formula in \eqref{eq:jacobian}.

Now fix any row $i$. Then 
\begin{align*}
  \sum_{j=1}^n (D_{\btheta} F)_{ij}
  &= \frac{\partial F_i}{\partial \theta_i}
    + \sum_{j \neq i} \frac{\partial F_i}{\partial \theta_j} \\
  &= -\sum_{j \neq i} K A_{ij}\cos(\theta_j - \theta_i)
    + \sum_{j \neq i} K A_{ij}\cos(\theta_j - \theta_i) = 0.
\end{align*}

Hence $D_{\btheta}F\,\bone = \bzero$, so $\bone \in \ker(D_{\btheta}F)$ and $\dim\ker(D_{\btheta}F) \geq 1$.  Finally, the matrix $D_{\btheta}F$  has the form of a weighted graph Laplacian with edge weights $K_c A_{ij}\cos((\theta_c)_j - (\theta_c)_i)$. Because $K_c$ is the \emph{critical} coupling, the Jacobian acquires an additional zero eigenvalue beyond the structural one from above. Thus, there exists a nonzero $\bv \notin \operatorname{span}\{\bone\}$ satisfying $D_{\btheta}F(\btheta_c, K_c, 0)\,\bv = \bzero$, completing the proof.

\end{proof}

To apply the Implicit Function Theorem (IFT), we track the critical null vector $\bv$ as the edge weight $a$ varies. We encode the steady-state equation, the criticality equation, a normalization, an orthogonality constraint to isolate the non-trivial null vector, and a gauge fix into a single augmented system.

\begin{definition}{Augmented System $G$}
Define $G : \R^n \times \R^n \times \R \times \R \to \R^{2n+3}$ by
\[
G(\btheta, \bv, K, a) =
\begin{pmatrix}
F(\btheta, K, a) \\
D_{\btheta}F(\btheta, K, a)\,\bv \\
\bv^{\top} \bv - 1 \\
\bv^{\top} \bone \\
\theta_1
\end{pmatrix}
= \bzero.
\]
\end{definition}

The five blocks serve distinct purposes:
\begin{itemize}
\item $F(\btheta, K, a) = \bzero$: the steady-state equation ($n$ equations).
\item $D_{\btheta}F(\btheta, K, a)\,\bv = \bzero$: the matrix-vector product tracking the non-trivial null vector of the Jacobian at criticality ($n$ equations).
\item $\bv^{\top}\bv - 1 = 0$: normalization of $\bv$ ($1$ equation).
\item $\bv^{\top}\bone = 0$: orthogonality of $\bv$ to the trivial Laplacian null vector $\bone$, isolating the non-trivial critical null vector ($1$ equation).
\item $\theta_1 = 0$: gauge condition fixing the rotational symmetry of the Kuramoto model ($1$ equation).
\end{itemize}
The free parameter is $a$; the solve-for variables are $(\btheta, \bv, K) \in \R^n \times \R^n \times \R = \R^{2n+1}$.

\begin{remark}
    At $a=0$ and at the critical point, $D_{\btheta}F$ has a two-dimensional kernel spanned by $\bone$ (trivial, present for all $K$) and $\bv$ (non-trivial, appearing only at $K = K_c$). The constraint $\bv^{\top}\bone = 0$ removes the trivial direction, while $\bv^{\top}\bv = 1$ removes the scaling ambiguity. Together they enforce $\dim\ker(D_{\btheta}F) = 1$ in the restricted sense needed for IFT.
\end{remark}
The partial Jacobian $D_{(\btheta,\bv,K)}G$ is a $(2n+3) \times (2n+1)$ matrix given by
\begin{equation}
D_{(\btheta,\bv,K)}G =
\begin{pmatrix}
D_{\btheta}F & 0_{n\times n} & \partial_K F \\[4pt]
D_{\btheta}(D_{\btheta}F\cdot\bv) & D_{\btheta}F & \partial_K(D_{\btheta}F)\bv \\[4pt]
0_{1\times n} & 2\bv^T & 0 \\[4pt]
0_{1\times n} & \bone^T & 0 \\[4pt]
\be_1^T & 0_{1\times n} & 0
\end{pmatrix},
\end{equation}
where 
\begin{equation}
\partial_K F = s(\btheta) + ag_{pq}(\btheta),
\end{equation}
and
\begin{equation}
\partial_K(D_{\btheta}F)\bv = \frac{1}{K}D_{\btheta}F\cdot\bv.
\end{equation}
The last identity follows from the fact that $D_{\btheta}F$ is linear in $K$, so that $\partial_K(D_{\btheta}F) = \frac{1}{K}D_{\btheta}F$. At the critical point $(\btheta_c, \bv_c, K_c, 0)$, we have $D_{\btheta}F\cdot\bv_c = \bzero$, so the second block of the third column vanishes identically. The domain ($2n+1$ variables) and codomain ($2n+3$ equations) have a two-dimensional mismatch, so IFT cannot be applied directly. We resolve this via an orthogonal projection.

\begin{lemma}[Cokernel of $D_{(\btheta,\bv,K)}G$]
\label{lem:cokernel}
The Jacobian $DG := D_{(\btheta,\bv,K)}G$, holding $a$ fixed, has a
two-dimensional left null space (cokernel) spanned by the orthonormal vectors
\[
  \psi_1 = \frac{1}{\sqrt{n}}(\bone, \bzero, 0, 0, 0)^T,
  \qquad
  \psi_2 = \frac{1}{\sqrt{n}}(\bzero, \bone, 0, 0, 0)^T,
\]
where $\bone, \bzero \in \R^n$.
\end{lemma}
\begin{proof} We show $\psi_1^T DG = 0$ and $\psi_2^T DG = 0$ then verify orthonormality. 

Now, the first block of $G$ is $F(\btheta, K, a)$. Since $\sum_i \omega_i=0$ and owing to the anti-symmetry of sine, the components of $F$ sum to zero for
\emph{all} $(\btheta, K, a)$:
\[
  \bone^T F(\btheta, K, a)
  = \sum_{i=1}^n \omega_i
    + K \sum_{i=1}^n \sum_{j=1}^n A_{ij}\sin(\theta_j - \theta_i)
    + Ka \sum_{i=1}^n g_{pq}(\btheta)_i.
\]
The first sum is zero since $\Omega$ is mean zero. For the second sum,
since $A_{ij} = A_{ji}$ and $\sin$ is odd, the $(i,j)$ and $(j,i)$ terms
cancel pairwise, giving zero. For the third sum, $g_{pq}(\btheta)
= \sin(\theta_p - \theta_q)(\be_q - \be_p)$, so
$\bone^T g_{pq}(\btheta) = \sin(\theta_p-\theta_q)(1 - 1) = 0$.
Hence $\bone^T F \equiv 0$ identically. Differentiating this identity with
respect to $(\btheta, \bv, K)$ gives
\[
  \psi_1^T DG
  = \frac{1}{\sqrt{n}}\,\bone^T D_{(\btheta,\bv,K)}G = 0.
\]

The second block of $G$ is $D_{\btheta}F(\btheta, K, a)\,\bv$. Since
$D_{\btheta}F$ is a symmetric matrix (the graph is undirected), its row
sums equal its column sums. We showed in Lemma~\ref{lem:kernel} that all
row sums are zero, so all column sums are also zero, i.e., $\bone^T D_{\btheta}F = 0^T$.
Therefore
\[
  \bone^T \bigl(D_{\btheta}F(\btheta, K, a)\,\bv\bigr)
  = \bigl(\bone^T D_{\btheta}F\bigr)\bv
  = 0
\]
for \emph{all} $(\btheta, \bv, K, a)$. Differentiating with respect to
$(\btheta, \bv, K)$ gives $\psi_2^T DG = 0$.  \\

For the orthonormality, note that 
\[
  \|\psi_1\|^2 = \frac{1}{n}\|\bone\|^2 = \frac{n}{n} = 1,
  \quad
  \|\psi_2\|^2 = \frac{1}{n}\|\bone\|^2 = 1,
  \quad
  \psi_1^T \psi_2 = 0,
\]
since the nonzero entries of $\psi_1$ and $\psi_2$ occupy disjoint blocks.
Hence $\{\psi_1, \psi_2\}$ is an orthonormal set in the left null space of
$DG$.

Under the non-degeneracy assumption that these are the only left-null
directions (verified numerically for the networks studied here), the cokernel
is exactly $\operatorname{span}\{\psi_1, \psi_2\}$.
\end{proof}

\begin{definition}[The Projection $P$]
\label{def:projection}
Since $\psi_1, \psi_2 \in \R^{2n+3}$ are orthonormal, define the orthogonal
projection onto their orthogonal complement by
\[
  P\bw = \bw - (\bw^T \psi_1)\psi_1 - (\bw^T \psi_2)\psi_2,
  \qquad \bw \in \R^{2n+3}.
\]
The image of $P$ is $(2n+1)$-dimensional and lies inside $\R^{2n+3}$.
\end{definition}

\begin{lemma}[Range Condition]
\label{lem:range}
For all $(\btheta, \bv, K, a)$, the vector $G(\btheta, \bv, K, a)$ lies in
$\ran(P)$, i.e., $\psi_i^T G = 0$ for $i = 1, 2$.
\end{lemma}
\begin{proof}
    Both identities were established in the proof of Lemma~\ref{lem:cokernel}: $\psi_1^T G = \frac{1}{\sqrt{n}}\bone^T F \equiv 0$ and $\psi_2^T G = \frac{1}{\sqrt{n}}\bone^T D_{\btheta}F\,\bv \equiv 0$, holding for all arguments, not only at solutions.
\end{proof}

\begin{corollary}
\label{cor:PGequivG}
$PG(\btheta, \bv, K, a) = \bzero$ if and only if $G(\btheta, \bv, K, a) = \bzero$.
\end{corollary}

\begin{proof}
    The forward direction is trivial: if $G = \bzero$ then $PG = P\bzero = \bzero$. For the converse, suppose $PG = \bzero$. By Definition~\ref{def:projection}, 
\begin{align}
    \mathbf{0} &= G - (G^T\psi_1)\psi_1 - (G^T\psi_2)\psi_2 \\
    G &= (G^T\psi_1)\psi_1 + (G^T\psi_2)\psi_2 \\
    G &= (\psi_1^TG)\psi_1 + (\psi_2^TG)\psi_2.
\end{align}
By Lemma~\ref{lem:range}, $\psi_i^T G = 0$ for $i = 1, 2$. Hence $G = 0 + 0 = \bzero$.
\end{proof}
\begin{remark}
    This is the key step that is sometimes called the "range condition" in Lyapunov--Schmidt reduction. The projection $P$ kills exactly the cokernel directions, and Corollary~\ref{cor:PGequivG} ensures that no information is lost by working with $H=PG$ instead of $G$ itself.
\end{remark}
After applying $P$, the first component (index $1$, the $\psi_1$-direction, corresponding to the mean of the $F$-block) and the $(n+1)$-st component (index $n+1$, the $\psi_2$-direction, corresponding to the mean of the $D_{\btheta}F\,\bv$-block) of $PG$ are identically zero. Dropping these redundant outputs yields a map
\[
  H(\btheta, \bv, K, a)
  := \text{(drop indices 1 and $n+1$ from } PG(\btheta, \bv, K, a)\text{)},
\]
so that $H : \R^{2n+2} \to \R^{2n+1}$. By Corollary~\ref{cor:PGequivG}, $H = \bzero$ if and only if $G = \bzero$. And now we are in a position to apply the IFT. 

\begin{lemma}[Conditions for IFT]
\label{lem:IFT}
Let $H(\btheta, \bv, K, a)$ be as above, so $H : \R^{2n+2} \to \R^{2n+1}$.
Treat $a \in \R$ as the free parameter and $(\btheta, \bv, K) \in \R^{2n+1}$
as the solve-for variables. Suppose:
\begin{enumerate}
  \item There exists a critical point $(\btheta_c, \bv_c, K_c)$ such that
    $H(\btheta_c, \bv_c, K_c, 0) = \bzero$.
  \item $H$ is continuously differentiable in a neighborhood of
    $(\btheta_c, \bv_c, K_c, 0)$.
  \item The partial Jacobian $D_{(\btheta,\bv,K)}H$ is a $(2n+1)\times(2n+1)$
    square matrix (domain and codomain dimensions match with $a$ held fixed).
  \item $D_{(\btheta,\bv,K)}H(\btheta_c, \bv_c, K_c, 0)$ is invertible. (This is a genericity condition, verified numerically for the networks studied here)
\end{enumerate}
Then there exist smooth functions $\btheta(a)$, $\bv(a)$, $K_c(a)$ defined for
$a$ near $0$ such that $\btheta(0) = \btheta_c$, $\bv(0) = \bv_c$,
$K_c(0) = K_c$, and $H(\btheta(a), \bv(a), K_c(a), a) = \bzero$ for all $a$
near $0$.
\end{lemma}

Conditions (i)--(iv) are precisely the hypotheses of the classical Implicit
Function Theorem applied to $H : \R^{2n+1} \times \R \to \R^{2n+1}$. The conclusion follows immediately from the standard IFT.

We now have all the ingredients to finalize the proof of Theorem~\ref{thm:saddlenode}. By Lemma ~\ref{lem:IFT}, there exist smooth functions $\btheta(a)$, $\bv(a)$,
$K_c(a)$ such that $H(\btheta(a),\bv(a),K_c(a),a) = \bzero$ for $a$ near $0$. 
By Corollary~\ref{cor:PGequivG}, $H = \bzero$ if and only if $G = \bzero$, so in particular the first block of $G$ gives $F(\btheta(a), K_c(a), a) = \bzero$ for all $a$ near $0$. 

Since $F(\btheta(a), K_c(a), a) = \bzero$ and all functions are smooth, we differentiate with respect to $a$ and evaluate at $a = 0$. By the chain rule, \begin{equation}
\label{eq:diffF}
  D_{\btheta}F\,\btheta'(0) + \partial_K F \cdot K'_c(0) + \partial_a F = \bzero.
\end{equation}
We compute each term separately.

For the term $\partial_K F$: differentiating $F(\btheta, K, a)$ with respect to $K$ gives: 
\[
  \partial_K F = s(\btheta) + a\,g_{pq}(\btheta),
\]
which at $a=0$ equals $s(\btheta(0))$.

For the term $\partial_a F$: differentiating with respect to $a$ gives
\[
  \partial_a F = K\,g_{pq}(\btheta),
\]
which at $a=0$ equals $K_c(0)\,g_{pq}(\btheta(0))$.

Substituting into \eqref{eq:diffF} at $a=0$:
\begin{equation}
\label{eq:diffF2}
  D_{\btheta}F(\btheta(0),K_c(0),0)\,\btheta'(0)
  + s(\btheta(0))\,K'_c(0)
  + K_c(0)\,g_{pq}(\btheta(0))
  = \bzero.
\end{equation}

Now multiply \eqref{eq:diffF2} on the left by $\bv(0)^T$:
\begin{align*}
  \bv(0)^T D_{\btheta}F(\btheta(0),K_c(0),0)\,\btheta'(0)
  + \bv(0)^T s(\btheta(0))\,K'_c(0)
  \\ + K_c(0)\,\bv(0)^T g_{pq}(\btheta(0))
  = 0.  
\end{align*}

Since $\bv(0)$ is in the kernel of $D_{\btheta}F(\btheta(0), K_c(0), 0)$ (this is the second block of $G=0$,  at $a=0$), the matrix $D_{\btheta}F$ is symmetric, so its kernel equals its left kernel:
\[
  \bv(0)^T D_{\btheta}F(\btheta(0),K_c(0),0) = 0. 
\]
Therefore, the first term vanishes, leaving
\[
  \bv(0)^T s(\btheta(0))\,K'_c(0)
  + K_c(0)\,\bv(0)^T g_{pq}(\btheta(0))
  = 0.
\]
By the genericity condition, $\bv(0)^T s(\btheta(0)) \neq 0$, we may solve:
\[K'_c(0)
  = -K_c(0)\,\frac{\bv(0)^T g_{pq}(\btheta(0))}{\bv(0)^T s(\btheta(0))}.\]
Recall that
\[g_{pq}(\btheta)=\sin(\theta_q-\theta_p)(e_p-e_q)\]
and at the steady state, when $a=0$,
\[0=\Omega+Ks(\theta)\implies s(\theta)=-\frac{\Omega}{K}.\]
Thus, 
\[\bv(0)^Tg_{pq}(\btheta(0))=\sin(\theta_q(0)-\theta_p(0))(v_p(0)-v_q(0))\]
and
\[\bv(0)^Ts(\theta)=-\bv(0)^T\frac{\Omega}{K_c(0)}.\]
Now expanding $\bv^T g_{pq}(\btheta)$ and $\bv^T s(\btheta)$ at $a=0$,
then substituting yields 
\begin{align*}
    K'_c(0)
& = \sin\!\bigl(\theta_q(0)-\theta_p(0)\bigr)\,(K_c(0))^2\,
    \frac{v_p(0) - v_q(0)}{\bv(0)^T\bOmega}. 
\end{align*}
This completes the proof.

\section{Proof of Theorem~\ref{thm:stablepert}} \label{appendixB}
\begin{proof}
Suppressing the dependence on $K$, we seek zeros of the following function $F(\mathbf{\theta},a)=0$ subject to the phase condition $\theta_1=0$.   Let $P:\mathbb{R}^{n}\to\mathbb{R}^{n-1}$ be the orthogonal projection defined via $(Pu)_{k-1}=(\mathbf{u}\cdot \mathbf{v}_k)$ for each eigenvector $\mathbf{v}_k$ with non-zero eigenvalue $\lambda_k$. We then define
\[ G(\mathbf{\theta},a)=\left(\begin{array}{c} PF(\mathbf{\theta},a) \\ \theta_1\end{array}\right).\]
Note that $G:\mathbb{R}^{n}\times\mathbb{R}^1\to\mathbb{R}^{n}$ and $D_\theta G$ is invertible since its kernel and cokernel are trivial.    Therefore, the Implicit Function Theorem gives the local existence of a function $\mathbf{\theta}^*(a)$ satisfying $\mathbf{\theta}^*(0)=\mathbf{\theta}^*$. 

Recall the definition of the order parameter $r(\mathbf{\theta})$; see  (\ref{eq:orderpar}).  Noting that we are considering $\theta$ as a function of the edge weight parameter $a$, this can be expressed as 
\[ r(\mathbf{\theta^*}(a))^2 =\frac{1}{n^2}\left(\sum_{j=1}^n \cos(\theta^*_j(a))\right)^2
+\frac{1}{n^2}\left(\sum_{j=1}^n \sin(\theta^*_j(a))\right)^2.\]
Define $C = \sum_{j=1}^n \cos\theta_j^*$ and $S = \sum_{j=1}^n \sin\theta_j^*$. Differentiating $r^2$ with respect to $a$ and evaluating at $a=0$ we find
\begin{align*}
    \frac{d}{da} \left.r(\theta^*(a))^2\right|_{a=0} & = \frac{2}{n^2}\left(\sum_j \cos\theta_j^*\right)
\left(-\sum_j \sin\theta_j^*\,\theta_j'(0)\right)\\ 
& + \frac{2}{n^2}\left(\sum_j \sin\theta_j^*\right)
\left(\sum_j \cos\theta_j^*\,\theta_j'(0)\right),
\end{align*}
which simplifies to 
\[ \frac{d}{da} \left.\frac{r(\theta^*(a))^2}{2}\right|_{a=0} = \frac{1}{n^2}\,\mathbf{w}\cdot\frac{d}{da}\mathbf{\theta}^*(a)|_{a=0},\]
where $\mathbf{w}$ has components $w_j = S\cos\theta_j^* - C\sin\theta_j^*$.

We have that 
\[ \Omega+Ks(\mathbf{\theta}^*(a))+Ka g_{pq}(\mathbf{\theta}^*(a))=0. \]
Differentiating with respect to $a$ and evaluating at $a=0$ yields

\[D_\theta F(\theta^*,0)\theta'(0)+D_aF(\theta^*,0)=0.\]
Since 
\[F(\theta,a)=\Omega+Ks(\theta)+Ka\,g_{pq}(\theta),\]
we have
\[D_aF(\theta^*,0)=K g_{pq}(\theta^*),\]
and therefore
\[D_\theta F(\theta^*,0)\theta'(0)=-Kg_{pq}(\theta^*).\]

Recalling the representation 
\[ g_{pq}=\sin(\theta_q-\theta_p)(\mathbf{e}_p-\mathbf{e}_q)\] 
with $p<q$. Since $\mathbf{1}\cdot g_{pq}=0$ and we
may apply the reduced inverse \[DF(\theta^*)^\dagger=\sum_{k=2}^n\frac1{\lambda_k}v_kv_k^T\] acting on the orthogonal complement of $v_1$. Then
\[\theta'(0)=-K\sin(\theta_q^*-\theta_p^*)DF(\theta^*)^\dagger(\mathbf{e}_p-\mathbf{e}_q),\]
which expands as 
\[\theta'(0)=-K\sin(\theta_q^*-\theta_p^*)\sum_{k=2}^n
\left(\frac{v_{kp}-v_{kq}}{\lambda_k}\right)v_k.\]
Substituting into the expression for $r(\theta(0))r'(\theta(0))$ and using $\gamma_k = \mathbf{w}\cdot v_k$ gives equation~\eqref{eq:ravaries}.
\end{proof}

\section{Theorem~\ref{thm:stablepert} as applied to Erd\"os-R\'enyi random graphs} \label{apc}
In Arola-Fernandez, Skardal and Arenas\cite{arola2021geometric}, a criterion for Braess paradox is derived in the context of random networks and applied to Erd\"os R\'enyi random graphs in particular.  This criterion applies for large $K$ where the following expansion is valid: 
\[ \theta_i^*=\frac{\omega_i}{Kd_i}+o\left(\frac{1}{K}\right),\]
where $d_i$ is the degree of the $i$th node. We then write 
\begin{equation}
\begin{split}  \frac{d}{da} &\left.\frac{r(\theta^*(a))^2}{2}\right|_{a=0} =-\frac{K_c(0)}{n^2} \mathbf{w} \cdot \left[DF(\mathbf{\theta}^*(0))^{-1} g_{pq}(\mathbf{\theta}^*(0))\right]  \\
&= -\sin(\theta^*_q(0)-\theta^*_p(0))\frac{K_c(0)}{n^2} \mathbf{w} \cdot \left[DF(\mathbf{\theta}^*(0))^{-1} \left(\mathbf{e}_p-\mathbf{e}_q\right)\right],
\end{split}
\end{equation}
where $w_j = S\cos\theta_j^* - C\sin\theta_j^*$ with $C = \sum_{j=1}^n \cos\theta_j^*$ and $S = \sum_{j=1}^n \sin\theta_j^*$, and the operator $DF^{-1}$ is understood as acting on the range of $DF(\mathbf{\theta}(0))$.
Following the approach in Arola-Fernandez et al.\cite{arola2021geometric}, the idea is to approximate 
\[DF(\mathbf{\theta}^*(0))\approx \mathrm{diag}\left(K_c(0)d_1,K_c(0)d_2,\dots, K_c(0)d_n\right), \]
and obtain, utilizing the fact that $\theta_i\approx \frac{\omega_i}{Kd_i}\ll 1$ to obtain
\[ \frac{d}{da} r(\theta^*(a))|_{a=0}\approx-\frac{1}{n^2}\left( \frac{\omega_q}{d_q}-\frac{\omega_p}{d_p}\right)\left(\frac{\omega_p}{d_p^2}-\frac{\omega_q}{d_q^2}\right).\]
We therefore recover from Theorem~\ref{thm:stablepert} the same prediction as in Arola-Fernandez et al.\cite{arola2021geometric} as the leading order term in the expansion for the derivative of $r(\theta^*)$ with respect to $a$ for $K$ large.

\end{appendix}

\bibliography{braessbib}

\end{document}